\date{October 6, 2009}

\documentclass[11pt]{amsart}
\usepackage{latexsym,amsmath,amsfonts,amscd,amssymb}
\usepackage{graphics}
\newtheorem{theorem}{Theorem}[section]
\newtheorem{lemma}[theorem]{Lemma}
\newtheorem{proposition}[theorem]{Proposition}
\newtheorem{corollary}[theorem]{Corollary}

\newtheorem{definition}[theorem]{Definition}

\newtheorem{example}[theorem]{Example}

\theoremstyle{remark}
\newtheorem{remark}[theorem]{Remark}

\newcommand{\la}{\langle}
\newcommand{\ra}{\rangle}
\newcommand{\too}{\longrightarrow}

\newcommand{\bd}{\partial}
\newcommand{\x}{\times}

\newcommand{\Vol}{\operatorname{Vol}}

\newcommand{\GL}{\operatorname{GL}}

\newcommand{\Hom}{\operatorname{Hom}}

\newcommand{\im}{\operatorname{im}}

\newcommand{\id}{\operatorname{id}}
\newcommand{\Hol}{\operatorname{Hol}}
\newcommand{\supp}{\operatorname{supp}}

\newcommand{\cA}{{\mathcal A}}

\newcommand{\cC}{{\mathcal C}}

\newcommand{\cM}{{\mathcal M}}

\newcommand{\cL}{{\mathcal L}}

\newcommand{\cS}{{\mathcal S}}
\newcommand{\cT}{{\mathcal T}}

\newcommand{\cV}{{\mathcal V}}

\newcommand{\CC}{{\mathbb C}}

\newcommand{\RR}{{\mathbb R}}

\newcommand{\ZZ}{{\mathbb Z}}

\renewcommand{\a}{\alpha}
\renewcommand{\b}{\beta}

\newcommand{\g}{\gamma}

\newcommand{\fro}{{\mathfrak{o}}}

\begin{document}

\title{Ergodic solenoids and generalized currents}


\author[V. Mu\~{n}oz]{Vicente Mu\~{n}oz}

\address{Facultad de
Matem\'aticas, Universidad Complutense de Madrid, Plaza de Ciencias
3, 28040 Madrid, Spain}

\email{vicente.munoz@mat.ucm.es}

\author[R. P\'{e}rez Marco]{Ricardo P\'{e}rez Marco}
\address{CNRS, LAGA UMR 7539, Universit\'e Paris XIII, 
99, Avenue J.-B. Cl\'ement, 93430-Villetaneuse, France}


\email{ricardo@math.univ-paris13.fr}

\thanks{Partially supported through Spanish MEC grant MTM2007-63582.
Second author supported by CNRS (UMR 7539)}

\maketitle

\begin{abstract}
We introduce the concept of solenoid as an abstract laminated space.
We do a thorough study of solenoids, leading to the notion of
ergodic and uniquely ergodic solenoids.
We define generalized currents associated with immersions of
oriented solenoids endowed with a transversal measure into smooth
manifolds, generalizing Ruelle-Sullivan currents.
\end{abstract}

\section{Introduction} \label{sec:introduction}

This is the first of a series of articles \cite{MPM2,MPM3,MPM4,MPM5}
in which we aim to give a geometric realization of {\em real}
homology classes in smooth manifolds, by using immersed laminations,
which we call {\em solenoids}. In this paper we define these
structures, we carry a thorough study, and we construct the homology
class associated to an oriented measured immersed solenoid in a
smooth manifold.

Let $M$ be a smooth compact connected and oriented manifold of
dimension $n\geq 1$ without boundary. Any closed oriented
submanifold $N\subset M$ of dimension $0\leq k\leq n$ determines a
homology class in $H_k(M, \ZZ)$. This homology class in
$H_k(M,\RR)$, as dual of De Rham cohomology, is explicitly given by
integration of the restriction to $N$ of differential $k$-forms on
$M$. Also, any immersion $f:N \to M$ defines an integer homology
class in a similar way by integration of pull-backs of $k$-forms.
Unfortunately, because of topological reasons dating back to Thom
\cite{Thom1}, not all integer homology classes in $H_k(M,\ZZ )$ can
be realized in such a way. Geometrically, we can realize any class
in $H_k(M, \ZZ)$ by topological $k$-chains. The real homology
$H_k(M,\RR)$ classes are only realized by formal combinations with
real coefficients of $k$-cells. This is not fully satisfactory.
In particular, for a variety of reasons (for example, in
the aim of developing a geometric intersection theory for real
homology classes), it is important to have an explicit realization,
as geometric as possible, of real homology classes.

In 1975, Ruelle and Sullivan \cite{RS} defined, for arbitrary
dimension $0\leq k\leq n$, geometric currents by using oriented
$k$-laminations embedded in $M$ and endowed with a transversal
measure. They applied their results to the stable and unstable
laminations of Axiom A diffeomorphisms (i.e. those with hyperbolic
non-wandering set with a dense set of periodic orbits). The point of
view of Ruelle and Sullivan is also based on duality. The
observation is that $k$-forms can be integrated on each leaf of the
lamination and then all over the lamination using the transversal
measure. This makes sense locally in each flow-box, and then it can
be extended globally by using a partition of unity. The result only
depends on the cohomology class of the $k$-form. It is natural to
ask whether it is possible to realize every real homology class
using a Ruelle-Sullivan current. A first result, with a precedent in
\cite{HM}, confirms that this is not the case: homology classes
with non-zero self-intersection cannot be represented by
Ruelle-Sullivan currents with no compact leaves (see Theorem
\ref{thm:self-intersection}).

More precisely, for each Ruelle-Sullivan lamination with a
non-atomic transversal measure, we can construct a smooth
$(n-k)$-form which provides the dual in de Rham cohomology (see section \ref{sec:forms}).
Using it, we prove that the self-intersection of a Ruelle-Sullivan current
(for a lamination) is zero, therefore it is not possible to represent a real homology
class in $H_k(M,\RR )$ with non-zero self-intersection (see section \ref{sec:embedded}).
This obstruction only exists when $n-k$ is even. This may be the
historical reason behind the lack of results on the representation
of an arbitrary homology class by Ruelle-Sullivan currents. In
section \ref{sec:Ruelle-Sullivan} we review and extend
Ruelle-Sullivan theory.

Therefore, in order to  represent every real homology class
we must first enlarge the class of Ruelle-Sullivan currents. This is
done by considering immersions of abstract oriented solenoids. We
define a $k$-solenoid to be a Hausdorff compact space foliated by
$k$-dimensional leaves with finite dimensional transversal structure
(see the precise definition in section \ref{sec:minimal}).

For these oriented solenoids we can consider $k$-forms that we can
integrate provided that we are given a transversal measure invariant
by the holonomy group. We define an {\it immersion} of a solenoid
$S$ into $M$ to be a regular map $f: S\to M$ that is an immersion in
each leaf. If the solenoid $S$ is endowed with a transversal measure
$\mu$, then any smooth $k$-form in $M$ can be pulled back to $S$ by
$f$ and integrated. The resulting numerical value only depends on
the cohomology class of the $k$-form. Therefore we have defined a
closed current that we denote by $(f,S_\mu )$ and call a generalized
current. This defines a homology class $[f,S_\mu] \in H_k(M,\RR )$.
Using these generalized currents, the above mentioned obstruction
disappears. Actually in \cite{MPM4}, we shall prove that every real
homology class in $H_k(M,\RR )$ can be realized by a generalized
current $(f,S_\mu)$ where $S_\mu$ is an oriented measured immersed
solenoid. Moreover in \cite{MPM5}, it is shown that the set of such
generalized currents $(f,S_\mu)$ realizing a given real homology
class $a\in H_k(M,\RR )$ is dense in the space of closed currents
representing $a$.

But the space of solenoids is large, and we would like to realize
the real homology classes by a minimal class of solenoids enjoying
good properties. We are first naturally led to topological
minimality. As we prove in section \ref{sec:minimal}, the spaces of
$k$-solenoids is inductive and therefore there are always minimal
$k$-solenoids. However, the transversal structure and the holonomy
group of minimal solenoids can have a rich structure, studied in
sections \ref{sec:topological} and \ref{sec:holonomy}. In particular,
such a solenoid may have many different transversal measures, each
one yielding a different generalized current for the same immersion
$f$. Therefore, of particular interest are uniquely ergodic solenoids,
with only one ergodic transversal measure. We study them in
section \ref{sec:transversal-structure}.

We also make a thorough study of Riemannian solenoids. We identify
transversal measures with the class of measures that disintegrate as
volume along leaves (daval measures), and also
prove a canonical decomposition of measures into a daval measure and
a singular part, corresponding to the classical Lebesgue
decomposition on manifolds (see section \ref{sec:riemannian}).


\noindent \textbf{Acknowledgements.} \
The authors are grateful to Alberto Candel, Etienne Ghys, Nessim Sibony,
Dennis Sullivan and Jaume Amor\'os for their comments and interest on
this work. In particular, Etienne
Ghys early pointed out on the impossibility of realization in
general of integer homology classes by embedded manifolds.
We thank the referee for a extremely careful reading of the manuscript
and many suggestions.

\section{Minimal solenoids} \label{sec:minimal}

We first define abstract solenoids, which are the main tool in this
article. As usual, $C^\omega$ denotes the space of analytic
functions. By $r\leq \omega$, we mean that $r$ is an integer, that $r=\infty$
or that $r=\omega$.

\begin{definition}\label{def:W}
  Let $0\leq s,r \leq \omega$, $r\geq s$, and let $k,\ell\geq 0$ be two integers.
  A {\rm foliated manifold} (of dimension $k+\ell$, with $k$-dimensional
  leaves, of regularity $C^{r,s}$) is a smooth
  manifold $W$ of dimension $k+\ell$ endowed with an atlas
  $\cA=\{ (U_i,\varphi_i)\}$, $\varphi_i:U_i\to \RR^k\x \RR^\ell$,
  whose transition maps
 $$
 \varphi_{ij}=\varphi_i\circ \varphi_j^{-1}: \varphi_j(U_i\cap U_j)
 \to \varphi_i(U_i\cap U_j) \ \, ,
 $$
are of the form $\varphi_{ij}(x,y)=(X_{ij}(x,y), Y_{ij}(y))$,
where $Y_{ij}(y)$ is of class $C^s$ and $X_{ij}(x,y)$ is of class
$C^{r,s}$.

A {\rm flow-box for} $W$ is a pair $(U,\varphi)$ consisting of an
open subset $U\subset W$ and a map $\varphi:U\to \RR^k\x \RR^\ell$
such that $\cA \cup \{(U,\varphi)\}$ is still an atlas for $W$.
\end{definition}

Clearly an open subset of a foliated manifold is also a foliated
manifold.

Given two foliated manifolds $W_1$, $W_2$ of dimension $k+\ell$,
with $k$-dimensional leaves, and of regularity $C^{r,s}$, a regular
map $f:W_1\to W_2$ is a continuous map which is locally, in
flow-boxes, of the form $f(x,y)=(X(x,y),Y(y))$, where $Y$ is of
class $C^{s}$ and $X$ is of class $C^{r,s}$.

A diffeomorphism $\phi:W_1\to W_2$ is a homeomorphism such that
$\phi$ and $\phi^{-1}$ are both regular maps.

\begin{definition}\label{def:k-solenoid}\textbf{\em ($k$-solenoid)}
Let  $0\leq r \leq s\leq \omega$, and let $k,\ell\geq 0$ be two
integers. A {\rm pre-solenoid} of dimension $k$, of class $C^{r,s}$
and transversal dimension $\ell$ is a pair $(S,W)$ where $W$ is a
foliated manifold and $S\subset W$ is a compact subspace which is a
union of leaves. 

Two pre-solenoids $(S,W_1)$ and $(S,W_2)$ are {\rm equivalent} if there
are open subsets $U_1\subset W_1$, $U_2\subset W_2$ with $S\subset
U_1$ and $S\subset U_2$, and a diffeomorphism $f: U_1\to U_2$ such
that $f$ is the identity on $S$.

A {\rm $k$-solenoid} of class $C^{r,s}$ and transversal dimension $\ell$
(or just a $k$-solenoid, or a solenoid) is an equivalence class of
pre-solenoids.
\end{definition}

We usually denote a solenoid by $S$, without making explicit
mention of $W$. We shall say that $W$ defines the solenoid
structure of $S$.

\begin{definition}\label{def:flow-box}
\textbf{\em (Flow-box)} Let $S$ be a solenoid. A {\rm flow-box} for
$S$ is a pair $(U,\varphi)$ formed by an open subset $U\subset S$
and a homeomorphism
 $$
 \varphi : U\to D^k\times K(U) \, ,
 $$
where $D^k$ is the $k$-dimensional open ball and $K(U)\subset
\RR^{\ell}$, such that there exists a foliated manifold $W$ defining
the solenoid structure of $S$, $S\subset W$, and a flow-box
$\hat\varphi:\hat U\to \RR^k\x\RR^\ell$ for $W$, with $U=\hat U\cap
S$, $\hat\varphi(U)= D^k\x K(U)\subset \RR^k\x\RR^\ell$ and
$\varphi=\hat\varphi_{|U}$.

The set $K(U)$ is the transversal space of the flow-box. The
dimension $\ell$ is the transversal dimension.
\end{definition}

\begin{figure}[h] 
\centering
\resizebox{8cm}{!}{\includegraphics{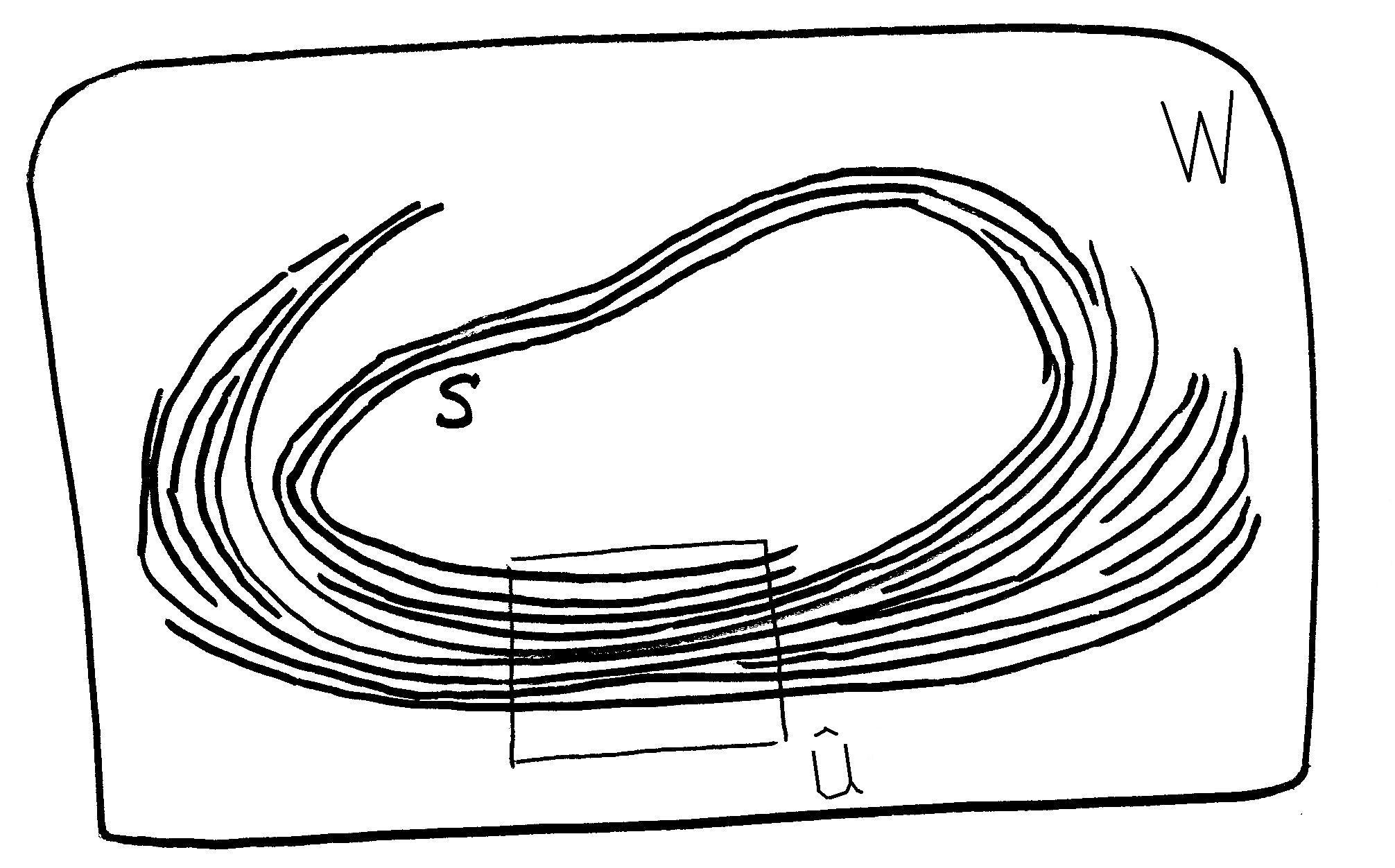}}    
\caption{Flow-box}
\end{figure}

As $S$ is locally compact, any point of $S$ is contained in a
flow-box $U$ whose closure $\overline U$ is contained in a bigger
flow-box. For such flow-box, $\overline U \cong \overline{D}^k
\times \overline K(U)$, where $\overline D^k$ is the closed unit
ball, $\overline K(U)$ is some compact subspace of $\RR^{\ell}$, and
$U=D^k \times  K(U) \subset \overline{D}^k \times \overline K(U)$.
We might call these flow-boxes \emph{good}. All flow-boxes that we
shall use are of this type so we shall not append any appelative to
them.

When the transversals of flow-boxes $K(U)\subset \RR^{\ell}$ are
open sets of $\RR^{\ell}$ we talk about full transversals. In this
case the solenoid structure is a $(k+\ell )$-dimensional compact
manifold foliated by $k$-dimensional leaves.

\begin{remark}
We refer to $k$ as the dimension of the solenoid and we write
 $$
 k=\dim S \, .
 $$
Note that, contrary to manifolds, this dimension in general does not
coincide with the topological dimension of $S$. The local structure
and compactness imply that solenoids are metrizable topological
spaces. The Hausdorff dimension of the transversals $K(U)$ is well
defined and obviously bounded by the transversal dimension $\ell$.
Thus, considering a finite covering by flow-boxes, we see that the
Hausdorff dimension of $S$, $\dim_H S$, is well defined, and equal
to
 $$
 \dim_H S=k+\max_U \dim_H K(U) \leq k+\ell <+\infty \, .
 $$
\end{remark}

\medskip

\begin{remark}
The definition of solenoid admits various generalizations. We
could focus on intrinsic changes of charts in $S$ with some
transverse Whitney regularity but without requiring a local
diffeomorphism extension. Such a definition would be more general,
but it is not necessary for our purposes. The present definition
balances generality and simplicity.

Another alternative generalization would be to avoid any
restrictive transversal assumption beyond continuity, and allow
for transversals of flow-boxes any topological space $K(U)$. But a
fruitful point of view is to regard the theory of solenoids as a
generalization of the classical theory of manifolds. Therefore it
is natural to restrict the definition only allowing finite
dimensional transversal spaces. For an alternative approach see
\cite{MoSch}.
\end{remark}

\medskip

\begin{definition}\label{def:diffeomorphisms}
\textbf{\em (Diffeomorphisms of solenoids)} Let $S_1$ and $S_2$ be
two  $k$-solenoids of class $C^{r,s}$ with the same transversal
dimension. A $C^{r,s}$-{\rm diffeomorphism} $f:S_1\to S_2$ is the
restriction of a $C^{r,s}$-diffeomorphism $\hat{f}:W_1\to W_2$ of
two foliated manifolds defining the solenoid structures of $S_1$ and
$S_2$, respectively.

\end{definition}

\begin{remark}
A homeomorphism of solenoids is a diffeomorphism of class $C^{0,0}$.
\end{remark}

This defines the category of smooth solenoids of a given regularity.
Note that we have the subcategory of smooth solenoids with full
transversals, and we have a forgetting functor into the category of
smooth manifolds.

\begin{definition} \label{def:leaf}\textbf{\em (Leaf)}
A {\rm leaf} of a $k$-solenoid $S$ is a leaf $l$ of any foliated
manifold $W$ inducing the solenoid structure of $S$, such that
$l\subset S$. Note that this notion is independent of $W$.
%
\end{definition}

Note that $S\subset W$ is the union of a collection of leaves.
Therefore, for a leaf $l$ of $W$ either $l\subset S$ or $l\cap
S=\emptyset$.

\medskip

Observe that when the transversals of flow-boxes $K(U)$ are totally
disconnected then the leaf-equivalence coincides with path
connection equivalence, and the leaves are the path connected
components of $S$.

\begin{definition} \label{def:oriented-solenoid}\textbf{\em (Oriented solenoid)}
An {\rm oriented solenoid} is a solenoid $S$ such that there is a
foliated manifold $W\supset S$ inducing the solenoid structure of
$S$, where $W$ has oriented leaves.
\end{definition}

It is easy to see that $S$ is oriented if and only if there is an
orientation for the leaves of $S$ such that there is a covering by
flow-boxes which preserve the orientation of the leaves.

Notice that we do not require $W$ to be oriented. For example, we
can foliate a M\"obius strip and create an oriented solenoid.

\begin{definition}\label{def:space-solenoids}
We define $\cS^{r,s}_{k,\ell}$ to be the space of $C^{r,s}$
$k$-solenoids with transversal dimension $\ell$.
\end{definition}

\begin{proposition}\label{prop:sub-solenoids}
Let $S_0\in \cS_{k,\ell}^{r,s}$ be a solenoid. A non-empty compact
subset $S$ of $S_0$ which is a union of leaves is a $k$-solenoid of
class $C^{r,s}$ and transversal dimension $\ell$.
\end{proposition}

\begin{proof}
Let $W$ be a $C^{r,s}$-foliated manifold defining the solenoid
structure of $S_0$. Then $S\subset W$ and $W$ defines a
$C^{r,s}$-solenoid structure for $S$.

Note that the flow-boxes of $S_0$ give, by restriction to $S$,
flow-boxes for $S$.
\end{proof}

\begin{corollary} \label{cor:connected-component}
Connected components of solenoids $\cS_{k,\ell}^{r,s}$ are in
$\cS_{k,\ell}^{r,s}$.
\end{corollary}

\begin{theorem}\label{thm:inductive}
The space $(\cS_{k,\ell}^{r,s}, \subset )$ ordered by inclusion is
an inductive set.
\end{theorem}

\begin{proof}
Let $(S_n) \subset \cS_{k,\ell}^{r,s}$ be a nested sequence of
solenoids, $S_{n+1} \subset S_n$. Define
 $$
 S_\infty = \bigcap_n S_n \, .
 $$
Then $S_\infty$ is a non-empty compact subset of $S_1$  as
intersection of a nested sequence of such sets. It is also a union
of leaves since each $S_n$ is so. Therefore by proposition
\ref{prop:sub-solenoids}, it is an element of $\cS_{k,\ell}^{r,s}$.
\end{proof}

\begin{corollary}\label{cor:minimal-elements}
The space $\cS_{k,\ell}^{r,s}$ has minimal elements.
\end{corollary}

\begin{proposition}\label{prop:dense}
If $S\in \cS_{k,\ell}^{r,s} $ is minimal then $S$ is connected. $S$
is minimal if and only if all leafs of $S$ are dense.
\end{proposition}

\begin{proof}
Each connected component of $S$ is a solenoid, thus by minimality
$S$ must be connected.

Also the closure ${\overline L}$ of any leaf $L\subset S$ is a
non-empty compact set union of leaves. Thus it is a solenoid and by
minimality we must have ${\overline L} =S$.

Conversely, if $S$ is not minimal, then there is a proper sub-solenoid
$S_0\subset S$. Take any leaf $l\subset S_0$. Then $l$ is not dense in $S$.
\end{proof}

\section{Topological transversal structure of solenoids}\label{sec:topological}

\begin{definition}\label{def:transversal}\textbf{\em (Transversal)}
Let $S$ be a $k$-solenoid. A {\rm local transversal} at a point
$p\in S$ is a subset $T$ of $S$ with $p\in T$, such that there is a
flow-box $(U, \varphi)$ of $S$ with $U$ a neighborhood of $p$
containing $T$ and such that
 $$
 \varphi (T)=\{0\}\x K(U) \, .
 $$

A {\rm transversal} $T$ of $S$ is a compact subset of $S$ such that for
each $p\in T$ there is an open neighborhood $V$ of $p$ such that
$V\cap T$ is a local transversal at $p$.
\end{definition}

If $S$ is a $k$-solenoid of class $C^{r,s}$, then any transversal
$T$ inherits an  $\ell$-dimensional $C^s$-Whitney structure.

We clearly have:

\begin{proposition}\label{prop:two-transversals}
The union of two disjoint transversals is a transversal.
\end{proposition}

\begin{definition}\label{def:global-transversal}
A transversal $T$ of $S$ is a {\rm global transversal} if all leaves
intersect $T$.
\end{definition}

The next proposition is clear.

\begin{proposition}\label{prop:global-transversals}
The union of two disjoint transversals, one of them global, is a
global transversal.
\end{proposition}

\begin{proposition}\label{prop:global-transversal}
If $S$ is minimal then all transversals are global. Moreover, if $S$
is minimal then any local transversal intersects all leaves of $S$.
\end{proposition}

\begin{proof}
It is enough to see the second statement, since it implies the
first. Let $U$ be a flow-box and $T=\varphi^{-1}(\{0\}\x K(U))$ a
local transversal (see definition \ref{def:flow-box}). By
proposition \ref{prop:dense}, all leaves intersect $U$ and therefore
they intersect $T$.
%
\end{proof}

Observe that the definition of solenoid with regular transverse
structure says that $S$ is always embedded in a
$(k+\ell)$-dimensional manifold $W$. Therefore $S\subset W$ has an
interior and a boundary relative to $W$. These sets do not depend on
the choice of $W$.

\begin{definition}\label{def:interior}\textbf{\em (Proper interior and boundary)}
Let $S$ be a $k$-solenoid. Let $W$ be a foliated manifold defining
the solenoid structure of $S$. The {\rm proper interior} of $S$ is
the interior of $S$ as a subset of $W$, considered as a
$(k+\ell)$-dimensional manifold (where $\ell$ is the transversal
dimension as usual).

The {\rm proper boundary} of $S$ is defined as the complement in $S$
of the proper interior.
\end{definition}

Let $\hat\varphi:\hat{U} \to \RR^k\x \RR^\ell$ be a flow-box for $W$
such that $U=\hat{U}\cap S$ and $\varphi=\hat\varphi_{|U}: U\to D^k
\x K(U)$ is a flow-box for $S$. Then $K(U)\subset \RR^\ell$. The
proper interior, resp.\ the proper boundary, of $S$, intersected
with $U$, consists of the collection of leaves $\varphi^{-1}(D^k \x
\{ y\})$, where $y\in K(U)$ is in the interior, resp.\ boundary, of
$K(U)\subset \RR^\ell$.

Note that the proper boundary of a solenoid that is a foliation of a
manifold is empty. We have the converse, as follows from proposition
\ref{prop:sub-solenoids}.

\begin{proposition}\label{prop:sub-solenoid}
If the proper boundary of $S$ is non-empty then it is a
sub-solenoid of $S$.
\end{proposition}


\begin{proposition}\label{thm:empty-interior}
Let $S\in \cS$ be a minimal solenoid. If $S$ is not the foliation of
a manifold then $S$ has empty proper interior, i.e. $K(U)\subset
\RR^\ell$ has empty interior for any flow-box $(U,\varphi )$.
\end{proposition}

\begin{proof}
The proper boundary of $S$ is non-empty  because otherwise, for each
flow-box $U$, $K(U)\subset \RR^\ell$ is an open set. Thus $S$ would
be an open subset of $W$, where $W$ is a foliated manifold defining
the solenoid structure of $S$, and so $S$ is itself a foliated
$(k+l)$-manifold. This contradicts the assumption.

Now by minimality the proper boundary must coincide with $S$ and the
proper interior is empty.
\end{proof}

\begin{example} \label{example}
 The dyadic solenoid is obtained as follows. Let $T=\bar{D}^2\times S^1$ be the solid torus, and consider
 the standard embedding $\imath: T\to \RR^3$. Let $f:T\to T$ be the embedding of $T$ into $T$ given
 by stretching the $\bar D^2$-direction and running over the $S^1$-direction twice (see Figure \ref{fi}, such
 $f$ is a hyperbolic map).
 Let $T_n=\imath \circ f^n (T)$, $n\geq 0$, and consider $S=\bigcap_{n\geq 0} T_n$. Then $S$ is a $1$-solenoid with
 $2$-dimensional transversal structure. This can be seen as follows: consider a smooth foliation on
 $T_0-T_1$ which is standard near $\bd T_0=S^1\times S^1$ (i.e. with leaves $\{p\}\times S^1$), and which is
 equal to the foliation $f(\bd T_0)$ on $\bd T_1$. We foliate $T_n-T_{n+1}$ by translating the
 foliation on $T_0-T_1$ via $f^n$. This gives a foliation on $T_0$, smooth on $T_0-S$, and of class $C^{\infty,0}$.
 So $S$ is a solenoid of class $C^{\infty,0}$.

 It is easy to see that $S$ is homeomorphic to the topological space
 $\underleftarrow{\mathrm{lim}}\ \{g^n:S^1 \to S^1\}$, where $g:S^1\to S^1$, $g(z)=z^2$.
 The above construction gives
 this space a solenoid structure.
\end{example}

\begin{figure}[h]
\centering
\resizebox{6cm}{!}{\includegraphics{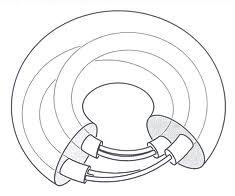}}    
\caption{The dyadic solenoid} \label{fi}
\end{figure}

Solenoids with a one dimensional transversal will play a prominent role
in \cite{MPM4}. We have for these the following structure theorem.

\begin{theorem} \label{thm:1-transversal}
\textbf{\em (Minimal solenoids with a $1$-dimensional
transversal).} Let $S\in \cS$ be a minimal $k$-solenoid which
admits a $1$-dimensional transversal $T$.

Then we have two cases:
 \begin{enumerate}
 \item {}$T$ is a finite union of circles, and $S$ is a $1$-dimensional foliation of a
 connected manifold of dimension $k+1$.
 \item {}$T$ is totally disconnected, in which case we have two further possibilities:
  \begin{enumerate}
   \item {}$T$ is a finite set and $S$ is a connected manifold of dimension $k$,
   \item {}$T$ is a Cantor set.
  \end{enumerate}
 \end{enumerate}
\end{theorem}

\begin{proof}
We define the proper interior of $T$ as the intersection of the
proper interior of $S$ with $T$. Now we have two cases.

If the proper interior of $T$ is non-empty, then the proper interior
of $S$ is non-empty. Then the complement of the proper interior of
$S$, if non-empty, is a sub-solenoid of $S$, contradicting
minimality. Thus the proper interior of $S$ is all $S$, so the
proper interior of $T$ is the whole of $T$. This means that any
point $p\in T$ has a neighborhood (in $T$) homeomorphic to an
interval. Therefore $T$ is a topological compact $1$-dimensional
manifold, thus a finite union of circles. This ends the first case.

If the proper interior of $T$ is empty, then $T$ is totally
disconnected. In this case, if $T$ has an isolated point $p$, then
$S$ has only one leaf because by minimality any other leaf must
accumulate the leaf containing $p$, and this is only possible if
it coincides with it. Then $S$ is a $k$-dimensional connected
manifold. If $T$ has no isolated points, then $T$ is non-empty,
perfect, compact and totally disconnected, i.e. it is a Cantor
set.
\end{proof}

\section{Holonomy, Poincar\'e return and suspension}\label{sec:holonomy}

We study in this section the holonomy properties of solenoids, some
of which are classical for foliations.

\begin{definition}\label{def:holonomy}\textbf{\em (Holonomy)}
Given two points $p_1$ and $p_2$ in the same leaf, two local
transversals $T_1$ and $T_2$, at $p_1$ and $p_2$ respectively, and a
path $\gamma :[0,1]\to S$, contained in the leaf with endpoints
$\g(0)=p_1$ and $\g(1)=p_2$, we define a germ of a map, the {\rm
holonomy map},
 $$
 h_\gamma: (T_1,p_1) \to (T_2, p_2)\, ,
 $$
by lifting $\gamma$ to nearby leaves.

We denote by ${\Hol}_S (T_1, T_2)$ the set of germs of holonomy
maps from $T_1$ to $T_2$.
\end{definition}


\begin{remark} If $T_1$ and $T_2$ are global transversals then
the sets of holonomy maps from $T_1$ to $T_2$ is non-empty. In
particular, if $S$ is minimal the set of holonomy maps between two
arbitrary local transversals is non-empty.
\end{remark}

\begin{definition}\label{def:holonomy-group} \textbf{\em (Holonomy pseudo-group)}
The {\rm holonomy pseudo-group} of a local transversal $T$ is the
pseudo-group of holonomy maps from $T$ into itself. We denote it by
${\Hol}_S(T)={\Hol}_S(T, T)$.

The holonomy pseudo-group of $S$ is the pseudo-group of all
holonomy maps. We denote it by ${\Hol}_S$,
 $$
 {\Hol}_S=\bigcup_{T_1 ,T_2} {\Hol}_S(T_1, T_2) \, .
 $$
\end{definition}


\begin{remark}  The orbit of a point
$x\in S$ by the holonomy pseudo-group coincides with the leaf
containing $x$.

Therefore, a solenoid $S$ is minimal if and only if the action of
the holonomy pseudo-group is minimal, i.e. all orbits are dense.
\end{remark}

The Poincar\'e return map construction reduces sometimes the
holonomy to a classical dynamical system.

\begin{definition}\label{def:Poincare}\textbf{\em (Poincar\'e return map)}
Let $S$ be an oriented minimal $1$-solenoid and $T$ be a local
transversal. Then the {holonomy return map} is well defined for
all points in $T$ and defines the {\rm Poincar\'e return map}
 $$
 R_T:T\to T \, .
 $$
\end{definition}

The return map is well defined because in minimal solenoids ``half''
leaves are dense.

\begin{lemma} \label{lem:half-leaves}
Let $S$ be a minimal $1$-solenoid. Let $p_0\in S$ and let
$l_0\subset S$ be the leaf containing $p_0$. The point $p_0$ divides
the leaf $l_0$ into two connected components. They are both dense in
$S$.
\end{lemma}

\begin{proof}
Consider one connected component of $l_0-\{p_0\}$, and let $C$ be
its accumulation set. Then $C$ is non-empty, by compactness of $S$,
and it is compact, as a closed subset of the compact solenoid $S$.
It is also a union of leaves because if $l\subset S$ is a leaf, then
$C\cap l$ is open in $l$ as is seen in flow-boxes, and also $C\cap
l$ is closed in $l$. Therefore by connectedness of $l$, $C\cap l$ is
empty or $l\subset C$.

We conclude that $C$ is a sub-solenoid, and by minimality we have
$C=S$.
\end{proof}

When $S$ admits a global transversal (in particular when $S$ is
minimal and admits a transversal) and the Poincar\'e return map is
well defined, we have
that it is continuous (without any assumption on minimality of $S$).

\begin{proposition}\label{prop:continuity}
Let $S$ be an oriented $1$-solenoid and let $T$ be a global transversal
such that the Poincar\'e return map $R_T$ is well defined. Then the
holonomy return map $R_T$ is continuous.
If the Poincar\'e return map for the reversed orientation of $S$ is
also well defined, then $R_T$ is a homeomorphism of $T$. Moreover, if $S$ is a
solenoid of class $C^{r,s}$ then $R_T$ is a $C^s$-diffeomorphism.
\end{proposition}

\begin{proof}
The map $R_T$ is continuous because
the inverse image of an open set is clearly open.

If the Poincar\'{e} return map $R_T^-$ for the same transversal obtained for
the reverse orientation of $S$ is also well defined, then
$R_T$ is bijective because by construction its inverse is $R_T^-$.
Hence $R_T$ is a
homeomorphism of $T$. Moreover, letting $W$ be a foliated manifold defining the solenoid structure of $S$,
$T$ is a subset of an open manifold $U$ of dimension $l$, and the
map $R_T$ extends as a homeomorphism $U_1\to U_2$, where $U_1$, $U_2$ are neighborhoods of $T$
(at least locally). If the
transversal regularity of $S$ is $C^s$ then the local extension of
$R_T$ is a $C^s$-diffeomorphism.
\end{proof}

\medskip

When $T$ is only a local transversal then in general $R_T$ is not
continuous. Nevertheless the discontinuities of $R_T$ are well
controlled in practice and are innocuous when we deal with measure
theoretic properties of $R_T$.

\medskip

The suspension construction reverses Poincar\'e construction of the
first return map.

\begin{definition}\label{def:suspension}\textbf{\em (Suspension construction)}
Let $X\subset \RR^\ell$ be a compact set and let $f:X\to X$ be a
homeomorphism which has a $C^s$-diffeomorphism extension to a
neighborhood of $X$ in $\RR^\ell$. The {\rm suspension} of $f$ is
the oriented $1$-solenoid $S_f$ defined by the suspension
construction
 $$
 S_f=([0,1]\times X)_{/(0,x)\sim (1,f(x))} \, .
 $$

\end{definition}

\begin{remark}
The solenoid $S_f$ has regularity $C^{\omega ,s}$ (the transition
maps are constructed with $f$).

The transversal $T=\{0 \}\times X$ is a global transversal and the
associated Poincar\'e return map $R_T:T\to T$ is well defined and
equal to $f$.
\end{remark}

In particular, the theory of dynamical systems for $X\subset \RR^\ell$
and diffeomorphisms $f:X\to X$ (extending to a neighborhood of $X$)
is contained in the theory of transversal structures of solenoids.

Note that example \ref{example} is a $1$-solenoid constructed by suspension.
The transversal $T$ is a Cantor set, homeomorphic to the $2$-adic integers
$\ZZ_2$, and the suspension map is $f:\ZZ_2\to \ZZ_2$, $n\mapsto n+1$.

\section{Measurable transversal structure of solenoids}\label{sec:transversal-structure}

In this section we study measure theory on solenoids, and in
particular the measurable transverse structure.

\begin{definition} \label{def:transversal-measure} \textbf{\em (Transversal measure)}
Let $S$ be a $k$-solenoid. A {\rm transversal measure} $\mu=(\mu_T)$
for $S$ is a collection of locally finite measures, each $\mu_T$
being associated to each local transversal $T$ and supported on $T$,
which are invariant by the holonomy pseudo-group (see definition
\ref{def:holonomy-group}). More precisely, if $T_1$ and $T_2$ are
two transversals and $h : V\subset T_1 \to T_2$ is a holonomy map,
then
 $$
 h_* (\mu_{T_1}|_{V})= \mu_{T_2}|_{h(V)} \, .
 $$
We assume that a transversal measure $\mu$ is non-trivial, i.e. for
some $T$, $\mu_T$ is non-zero.

We denote by $S_\mu$ a $k$-solenoid $S$ endowed with a transversal
measure $\mu=(\mu_T)$. We refer to $S_\mu$ as a {\rm measured solenoid}.
\end{definition}

Observe that for any transversal measure $\mu=(\mu_T)$ the scalar
multiple $c\, \mu=(c \, \mu_T)$, where $c>0$, is also a transversal
measure. Notice that there is no natural scalar normalization of
transversal measures.

\begin{definition}\label{def:support} \textbf{\em (Support of a transversal measure)}
Let $\mu=(\mu_T)$ be a transversal measure. We define the {\rm support
of $\mu$} by
 $$
 \supp \mu=\bigcup_T \supp \mu_T \, ,
 $$
where the union runs over all local transversals $T$ of $S$.
\end{definition}

\begin{proposition}\label{prop:support}
The support of a transversal measure $\mu$ is a sub-solenoid of $S$.
\end{proposition}

\begin{proof}
For any flow-box $U$, $\supp \mu\cap U$ is closed in $U$, since
$\supp \mu_{K(U)}$ is closed in $K(U)$. Hence, $\supp \mu$ is closed
in $S$. Also, locally in
flow-boxes $\supp \mu$ contains full leaves of $U$. Therefore a
leaf of $S$ is either disjoint from $\supp \mu$ or contained in
$\supp  \mu$. Also $\supp \mu $ is non-empty because $\mu$ is
non-trivial. We conclude that $\supp \mu$ is a sub-solenoid.
\end{proof}

\begin{definition} \label{def:transversal-ergodicity}\textbf{\em (Transverse ergodicity)}
A transversal measure $\mu=(\mu_T )$ on a solenoid $S$ is {\rm ergodic} if for any Borel set
$A\subset T$ invariant by the pseudo-group of holonomy maps on $T$,
we have
 $$
 \mu_T(A) = 0 \ \ {\hbox{\rm{ or }}} \ \ \mu_T(A) = \mu_T(T) \, .
 $$

We say that $S_\mu$ is an {\rm ergodic solenoid}.
\end{definition}

\begin{definition} \label{def:transversal-unique-ergodicity}
\textbf{\em (Transverse unique ergodicity)} Let $S$ be a
$k$-solenoid. The solenoid $S$ is {\rm transversally uniquely ergodic}, or
a {\rm uniquely ergodic solenoid}, if $S$ has a unique up to scalars
transversal measure $\mu$ and moreover $\supp \mu =S$.
\end{definition}

Observe that in order  to define these notions we only need
continuous transversals. These ergodic notions are intrinsic and
purely topological, i.e. if $S_{1}$ and $S_{2}$ are two homeomorphic solenoids by a
homeomorphism $h:S_1\to S_2$, then $S_1$ is uniquely ergodic if and only if $S_2$ is.
If $S_{1,\mu_1}$ and $S_{2,\mu_2}$ are homeomorphic and $\mu_2=h_* \mu_1$ via the
homeomorphism $h:S_1\to S_2$, then
$S_{1,\mu_1}$ is ergodic if and only if $S_{2,\mu_2}$ is.

\medskip

These notions of ergodicity generalize the classical ones and do
exactly correspond to the classical notions in the situation described
by the next theorem.

\begin{theorem}\label{thm:ergodic}
Let $S$ be an oriented $1$-solenoid. Let $T$ be a global transversal
such that the Poincar\'e return map $R_T: T\to T$ is well defined.

Then the solenoid $S$ is minimal, resp.\ ergodic, uniquely ergodic,
if and only if $R_T$ is minimal, resp.\ ergodic, uniquely ergodic.
\end{theorem}

\begin{proof}
We have by proposition \ref{prop:continuity} that $R_T$ is
continuous. A leaf of $S$ is dense if and only if its intersection
with $T$ is a dense orbit of $R_T$, hence the equivalence of
minimality.

For the ergodicity, observe that we have a correspondence between
measures on $T$ invariant by $R_T$ and transversal measures for $S$.
Each transversal measure for $S$, locally defines a measure on $T$,
hence defines a measure on $T$. Conversely, given a measure $\nu$ on
$T$, we can transport $\nu$ in order to define a measure in each
local transversal $T'$ in the following way. We can define a map
$R_{T',T}: T'\to T$ of first impact on $T$ by following leaves of
$S$ from $T'$ in the positive orientation. By the global character
of the transversal this map is well defined. By construction
$R_{T',T}$ is injective. So we can define
$\mu_{T'}=R_{T',T}^*\nu_{R_{T',T}(T')}$. Then $(\mu_{T'})$ defines a
transversal measure. The equivalence of unique ergodicity follows.
Also $\nu$ is ergodic if and only $(\mu_{T'})$ is ergodic because any
decomposition of $\nu=\nu_1+\nu_2$ induces a decomposition of
$(\mu_{T'})$ by the transversal measures corresponding to the
decomposing measures.
\end{proof}

When we have an ergodic oriented $1$-solenoid $S_\mu$ and $T$ is a
local transversal, then the Poincar\'e return map is well defined
$\mu_T$-almost everywhere and $\mu_T$ is ergodic.

\begin{proposition}\label{prop:5.7}
Let $S$ be an oriented $1$-solenoid and let $T$ be a local
transversal of $S$. Let $\mu$ be an ergodic transversal measure for
$S$. Then the Poincar\'e return map $R_T$ is well defined for
$\mu_T$-almost all points of $T$ and $\mu_T$ is an ergodic measure
of $R_T$.
\end{proposition}

\begin{proof}
Let $A_T\subset T$ be the set of wandering points of $T$, i.e. those
points whose positive half leaves through them never meet $T$ again.
Clearly $A_T$ is a Borel set. If $\mu_T (A_T)\not= 0$ we can
decompose $\mu_T$ by decomposing $\mu_{T |A_T}$ and transporting the
decomposition (back and forward) by the holonomy in order to
decompose the transversal measure. Therefore $\mu_T (A_T)=0$. As
before, a decomposition of $\mu_T$ into invariant measures by $R_T$
yields a decomposition of the transversal measure $\mu$ invariant by
holonomy.
\end{proof}

Recall that a dynamical system is minimal when all orbits are dense,
and that uniquely ergodic dynamical systems are minimal. We have the
same result for uniquely ergodic solenoids.

\begin{proposition}\label{cor:ergodic-1-solenoid}
An oriented uniquely ergodic $1$-solenoid $S$ is minimal.
\end{proposition}

\begin{proof}
If $S$ has a non-dense leaf $l\subset S$, we can consider a local
transversal $T_0$ such that $T_0\cap \bar l \not= \emptyset$. Let
$(l_n)$ be an exhaustion of $l$ by compact subsets. Let $\mu_n$ be
the atomic probability measure on $T_0$ equidistributed on the
intersection of $l_n$ with $T_0$. Any limit measure $\mu_{n_k}\to
\nu$ is a probability measure on $T_0$ with $\supp \nu \subset
T_0\cap \bar l$. It follows easily that $\nu$ is invariant by the
holonomy on $T_0$. Transporting by the holonomy, $\nu$ defines a
transversal measure $\mu=(\mu_T)$ (up to normalization, in each
transversal it is also a limit of counting measures). But this
contradicts unique ergodicity since $\supp \mu \not= S$.
\end{proof}

Given a measured solenoid $S_\mu$ we can talk about ``$\mu$-almost
all leaves'' with respect to the transversal measure. More
precisely, a Borel set of leaves has zero $\mu$-measure if the
intersection of this set with any local transversal $T$ is a set of
$\mu_T$-measure zero.

Now Poincar\'e recurrence theorem for classical dynamical systems
translates as:

\begin{proposition} \label{cor:Poincare-recurrence} \textbf{\em (Poincar\'e recurrence)}
Let $S_\mu$ be an ergodic oriented $1$-solenoid with $\supp \mu =S$.
Then $\mu$-almost all leaves are dense and accumulate themselves.
\end{proposition}

\begin{proof}
For each local transversal $T$ we know by proposition \ref{prop:5.7}
that the Poincar\'e return map $R_T$ is
defined for $\mu_T$-almost every point and leaves invariant $\mu_T$.
Therefore by Poincar\'e recurrence theorem, $\mu_T$-almost every
point has a dense orbit by $R_T$ in $\supp \mu_T=T$.

Observe that $S_\mu$ ergodic implies that $S$ is connected
(otherwise we may decompose the invariant measure by restricting it to each
connected component).

By compactness we can choose a finite number of local transversals
$T_i=\varphi^{-1}(\{ 0\}\x K(U_i))$ with flow-boxes $\{U_i\}$
covering $S$. We can assume that we have that $U_i\cap U_j$ is a
flow-box if non-empty. This assumption and the connectedness of $S$
imply that any Borel set of leaves that has either total or zero
measure in a flow-box $U_i$, has the same property in $S$.

Now, the set of leaves $S_i$ that are non-dense in a given
flow-box $U_i$ is of zero $\mu$-measure in $U_i$ (by Poincar\'e
recurrence theorem applied to $R_{T_i}$). By the above,
$S_i$ is of zero $\mu$-measure in $S$. Finally the set of
non-dense leaves in $S$ is the finite union of the $S_i$,
therefore is a set of leaves of zero $\mu$-measure.
\end{proof}

We denote by $\cM_\cT (S)$ the space of transversal positive measures
on the solenoid $S$ equipped with the topology generated by weak
convergence in each local transversal.

We denote by $\overline\cM_\cT (S)$ the quotient of $\cM_\cT(S)$
by positive scalar multiplication.

\begin{proposition}\label{prop:extremal-ergodic-transversal}
The space $\cM_\cT (S)$ is a cone in the vector space of all
transversal signed measures $\cV_\cT (S)$. 
Extremal measures of $\cM_\cT (S)$ correspond to ergodic tranversal measures.
\end{proposition}

\begin{proof}
Only the last part needs a proof. If $(\mu_T)$ is not ergodic, then
there exists a local tranversal $T_0$ and two disjoint Borel set
$A,B\subset T_0$ invariant by holonomy with $\mu_{T_0}(A)\not=0$,
$\mu_{T_0}(B)\not=0$ and
$\mu_{T_0}(A)+\mu_{T_0}(B)=\mu_{T_0}({T_0})$. Let $S_A\subset S$,
resp.\ $S_B\subset S$, be the union of leaves of the solenoid
intersecting $A$, resp.\ $B$. These are Borel subsets of $S$. Let
$(\mu_{T |T\cap S_A})$ and $(\mu_{T |T\cap S_B})$ be the transversal
measures conditional to $T\cap S_A$, resp.\ $T\cap S_B$. Then
 $$
 (\mu_T)=(\mu_{T |T\cap S_A})+(\mu_{T |T\cap S_B}) \, ,
 $$
and $(\mu_T)$ is not extremal.
\end{proof}

\begin{corollary} \label{cor:vaya-lio}
If $\cM_\cT (S)$ is non-empty then $\cM_\cT (S)$ contains ergodic
measures.
\end{corollary}

We shall provide the proof of this result after theorem \ref{thm:transverse-riemannian}.

\section{Riemannian solenoids}\label{sec:riemannian}

In this section we endow solenoids with a Riemannian structure and
we study their metric properties.

All measures considered are Borel measures and all limits of
measures are understood in the weak-* sense. We denote by $\cM (S)$
the space of probability measures supported on $S$.

\begin{definition} \label{def:riemannian}\textbf{\em (Riemannian solenoid)}
Let $S$ be a $k$-solenoid of class $C^{r,s}$ with $r\geq 1$. A
{\rm Riemannian structure on $S$} is a Riemannian metric on the leaves
of $S$ such that there is a foliated manifold $W$ defining the
solenoid structure of $S$ and a metric $g_W$ on the leaves of $W$ of
class $C^{r-1,s}$ such that $g=g_{W|S}$.
\end{definition}

For instance, example \ref{example} can be given a Riemannian structure by restricting the
Riemannian metric of $\RR^3$ to the leaves of the solenoid.

As for compact manifolds, via a partition of unity, any solenoid can
be endowed with a Riemannian structure.

In the rest of this section
$S$ denotes a Riemannian solenoid.
Note that a Riemannian structure defines a $k$-volume on the leaves
of $S$. This is a function $\Vol_k$ which assigns to any subset $A\subset l$
on a leaf $l\subset S$ its volume with respect to the Riemannian metric on
the leaf.

\begin{definition}\textbf{\em (Flow group)} \label{def:flow-group}
We define the {\rm flow group} $G_S^0$ of a Riemannian $k$-solenoid $S$ as
the group of $k$-volume preserving diffeomorphisms of $S$ isotopic to
the identity in the group of diffeomorphims of $S$. We define the
{\rm extended flow group} $G_S$ as the group of $k$-volume preserving diffeomorphisms
of $S$. 
\end{definition}

Note that we do not claim that $G_S^0$ is the connected component of the identity
of $G_S$, although this may well be true.

\begin{definition}\label{def:desintegrate}\textbf{\em (daval
measures)} Let $\mu$ be a measure supported on $S$. The measure
$\mu$ is a {\rm daval measure} if it disintegrates as volume along leaves
of $S$, i.e.\ for any flow-box $(U,\varphi)$ with local transversal
$T=\varphi^{-1}(\{0\}\x K(U))$, we have a measure $\mu_{U,T}$
supported on $T$ such that for any Borel set $A\subset U$
 $$
 \mu(A)=\int_{T} {\Vol}_k(A_y) \ d\mu_{U,T}(y) \, ,
 $$
where $A_y=A\cap \varphi^{-1} (D^k\times \{ y \} )\subset U$.

We denote by $\cM_\cL (S)\subset \cM (S)$ the space of probability
daval measures.
\end{definition}

It follows from this definition that the measures $(\mu_{U,T})$ do
indeed define a transversal measure as we prove in the next
proposition.

\begin{proposition} \label{prop:desintegrate} Let $\mu$ be a daval measure on $S$.
Then we have the following properties.
 \begin{enumerate}
 \item[(i)] For a local transversal $T$, the measures $\mu_{U,T}$ do not depend on
 $U$. So they define a unique measure $\mu_T$ supported on $T$.
 \item[(ii)] The measures $(\mu_T)$ are uniquely determined by $\mu$.
 \item[(iii)] The measures $(\mu_T)$ are locally finite.
 \item[(iv)] The measures $(\mu_T)$ are invariant by holonomy and
 therefore define a transversal measure.
 \end{enumerate}
\end{proposition}

\begin{proof}
For (i) and (ii) notice that for any Riemannian metric $g$ we have,
denoting by $B_\epsilon^g (y)$ the Riemannian ball of radius $\epsilon$ around $y$
in its leaf,
 $$
 \lim_{\epsilon \to 0} \frac{\Vol_k(B_\epsilon^g (y))}
 {\epsilon^k}=c(k) \, ,
 $$
where $c(k)$ is a constant only depending on $k$. Therefore by
dominated convergence we have for any Borel set $C\subset T$
 $$
 \mu_{U,T}(C)=\lim_{\epsilon \to 0} \int_T
 \frac{{\Vol}_k(B_\epsilon^g (y))}{c(k) \epsilon^k} \ d\mu_{U,T}(y)
 =\lim_{\epsilon \to 0} \frac{\mu(V_{\epsilon} (C))}{c(k) \epsilon^k}
 \, ,
$$
where $V_\epsilon(C)$ denotes the $\epsilon$-neighborhood of $C$ along
leaves. The last limit is clearly independent of $U$, thus
$\mu_{U,T}$ is independent of $U$ as claimed, and $\mu_T$ is
uniquely determined by $\mu$.

\medskip

For (iii) observe that for each flow-box $U$ we have that
 $$
 y\mapsto {\Vol}_k (L_y)\, ,
 $$
$L_y=\varphi^{-1}(D^k\x\{y\})$, is continuous on $y\in T$, therefore
for any compact subset $C\subset T$ exists $\epsilon_0
>0$ such that for all $y\in C$,
 $$
 {\Vol}_k (L_y) \geq \epsilon_0 \, .
 $$
Let $V=\varphi^{-1}(D^k\x C)$, then we have
 \begin{equation}\label{eqn:formula}
 \mu (V)=\int_C {\Vol}_k (L_y) \ d\mu_T(y) \geq \epsilon_0 \
 \mu_T (C) \, ,
 \end{equation}
therefore $\mu_T(C) <+\infty $.

\medskip

Regarding (iv), consider a flow-box $(U,\varphi)$ and two local
transversals $T_1$ and $T_2$ in $U$ of the form
$T_i=\varphi^{-1}(\{x_i\} \x K(U))$, $i=1,2$, $x_i\in D^k$. These
transversals are associated to flow-boxes $(U,\varphi_i)$ with the
same domain $U$. There is a local holonomy map in $U$, $h:T_1 \to
T_2$. For any Borel set $A\subset U$, we have by definition
 $$
 \int_{T_1} {\Vol}_k(A_y) \ d\mu_{U,T_1}(y)=\mu(A)=\int_{T_2}
 {\Vol}_k(A_{y'}) \ d\mu_{U,T_2}({y'}) \, .
 $$
On the other hand, the change of variables, $y'=h(y)$, gives
 $$
 \int_{T_1} {\Vol}_k(A_y) \ d\mu_{U,T_1}(y)=\int_{T_2}
 {\Vol}_k(A_{y'}) \ dh_* \mu_{U,T_1}(y') \, .
 $$
Thus for any Borel set $A\subset U$,
 $$
 \int_{T_2} {\Vol}_k(A_{y'}) \ d\mu_{U,T_2}({y'})=\int_{T_2}
 {\Vol}_k(A_{{y'}}) \ dh_* \mu_{U,T_1}({y'}) \, .
 $$
And taking horizontal Borel sets, this implies
 $$
 \mu_{U,T_2}=h_* \mu_{U,T_1} \, .
 $$

The invariance by local holonomy implies the invariance by all
holonomies. Take two arbitrary local transversals $T_1'$ and $T_2'$,
two points $p_1\in T_1'$, $p_2\in T_2'$ in the same leaf, and a path
$\gamma$ from $p_1$ to $p_2$ inside a leaf. Then we can construct a
small neighborhood flow-box $(U,\varphi)$ around the curve $\gamma$,
so that $T_1''=\varphi^{-1}(\{x_1\} \x K(U))\subset T_1'$ and
$T_2''=\varphi^{-1}(\{x_2\} \x K(U))\subset T_2'$ ($x_1$ and $x_2$
being two distinct points of $D^k$) are open subsets of the
respective transversals and $\gamma$ is fully contained in a leaf of
$U$.
\end{proof}

{}From this it follows that Riemannian solenoids do not necessarily
have daval measures (i.e. $\cM_\cL(S)$ can be empty), because there
are solenoids which do not admit transversal measures (see \cite{Pl}
for interesting examples).

\begin{proposition} \label{prop:M_L-compact}
The space of probability daval measures $\cM_\cL (S)$ is a compact
convex set in the vector space of signed measures $\cV(S)$.
\end{proposition}

\begin{proof}
The convexity is clear, and by compactness of $\cM(S)$ we only need
to show that $\cM_\cL (S)$ is closed, which follows from the more
precise lemma that follows.
\end{proof}

\begin{lemma}\label{lem:desintegrate}
Let $(\mu_n)$ be a sequence of measures on $S$ that disintegrate as
volume on leaves in a flow-box $U$. Then any limit $\mu$
disintegrates as volume on leaves in $U$ and the transversal measure
is the limit of the transversal measures.
\end{lemma}

\begin{proof}
We assume that $\mu_n\to \mu$. Given the transversal $T$, the
transversal measures $(\mu_{n,T})$ are locally finite by proposition
\ref{prop:desintegrate}. Moreover, formula (\ref{eqn:formula}) shows
that they are uniformly locally finite. Extract (with a diagonal
process) a converging subsequence $\mu_{n_k,T}$ to a locally finite
measure $\mu_T$. For any vertically compactly supported continuous
function $\phi$ defined on $U$ and depending only on $y\in T$,
 $$
 \int_U \phi \ d\mu_{n_k}=\int_T \phi \Vol_k(L_y)\ d\mu_{n_k,T}(y)\, ,
 $$
with $L_y=\varphi^{-1}(D^k\x \{y\})$. Passing to the limit $k\to
+\infty$,
 $$
 \int_U \phi \ d\mu =\int_T \phi \Vol_k(L_y)\ d\mu_{T}(y) \, .
 $$
Therefore the limit measure $\mu$ disintegrates as volume on leaves
in $U$ with transversal measure $\mu_T$. Since $\mu_T$ is uniquely
determined by $\mu$ (by proposition \ref{prop:desintegrate}), the
only limit of the transversal measures is $\mu_T$.
\end{proof}

\begin{theorem}\label{thm:G-invariance}
A finite measure $\mu$ on $S$ is $G_S^0$-invariant if and only if
$\mu$ is a daval measure.
\end{theorem}

\begin{proof}
If $\mu$ disintegrates as volume along leaves, then it is clearly
invariant by a transformation in $G_S^0$ close to the identity as is
seen in each flow-box. Then it is $G_S^0$-invariant since a
neighborhood of the identity in $G_S^0$ generates $G_S^0$.

Conversely, assume that $\mu$ is a $G_S^0$-invariant finite measure.
We must prove that in any flow-box $(U,\varphi)$ we have $\mu=\Vol_k
\x \mu_{K(U)}$. We can find a map $h: D^k\x K(U) \to \RR^k\x K(U)$
of class $C^{r,s}$, preserving leaves, and such that it takes the
$k$-volume for the Riemannian metric to the Lebesgue measure on
$\RR^k$. On $h(D^k\x K(U))$, we still denote by $\mu$ the
corresponding measure. We can disintegrate $\mu= \{\nu_y\} \x \eta$,
where $\eta$ is supported on $K(U)$ and $\nu_y$ is a measure on each
horizontal leaf, parametrized by $y\in K(U)$ (see \cite{Die} for
disintegration of measures), i.e.
 \begin{equation}\label{eqn:eta}
 \int_U \phi \ d\mu=\int_{K(U)} \left ( \int_{L_y} \phi \
 d\nu_y\right ) \ d\eta (y)\, .
 \end{equation}

The group $G_S^0$ in this chart contains all small translations.
Each translation must leave invariant $\eta$-almost all measures
$\nu_y$. Therefore a countable number of translations leave
invariant $\eta$-almost all measures $\nu_y$. Now observe that if $\tau_n$ are
translations leaving invariant $\nu_y$,  and $\tau_n\to \tau$, then $\tau$
leaves $\nu_y$ invariant. Thus taking a countable and dense set of
translations of fixed small displacement, and taking limits, it
follows that all small translations leave invariant $\nu_y$ for
$\eta$-almost all $y$. By Haar theorem these measures are
proportional to the Lebesgue measure, $\nu_y=c(y){\Vol}_k$. We have
that $c\in L^1(K(U),\eta)$ by applying (\ref{eqn:eta}) with
$\varphi$ being the characteristic function of a sub-flow-box with
horizontal leaves being balls of fixed $k$-volume. We define the
transversal measure $\mu_{K(U)}$ as
 $$
 d\mu_{K(U)} =c \ d\eta \, .
 $$
Therefore $\mu=\Vol_k \times \mu_{K(U)}$ on $U$, hence $\mu$ is a daval measure.
\end{proof}

\begin{theorem} \label{thm:transverse-riemannian}
\textbf{\em (Tranverse measures of the Riemannian solenoid)} There
is a one-to-one correspondence between transversal measures $(\mu_T)$
and finite daval measures $\mu$. Furthermore, there is an
isomorphism
  $$
  \overline{\cM}_\cT(S) \cong \cM_\cL(S)\, ,
  $$
between the space $\overline\cM_\cT (S)$ of transversal positive
measures on $S$ modulo positive scalar multiplication, and the
space $\cM_\cL(S)$ of probability daval measures on $S$.
\end{theorem}

\begin{proof}
The open sets inside flow-boxes form a basis for the Borel
$\sigma$-algebra, and the formula
 $$
 \mu(A)=\int_T {\Vol}_k(A_y) \ d\mu_{T}(y) \, ,
 $$
for $A$ in a flow-box $U$ with local transversal $T$, is compatible
for different flow-boxes. So it defines a measure $\mu$. This
measure is finite because by compactness we can cover $S$ by a
finite number of flow-boxes with finite mass. By construction, $\mu$
is a daval measure. The converse was proved earlier in proposition
\ref{prop:desintegrate}.

This correspondence is clearly a topological isomorphism.
\end{proof}

\medskip

\noindent {\em Proof of corollary \ref{cor:vaya-lio}} \
First, to prove that $\cM_\cT(S)$ has ergodic measures is
equivalent to proving that $\overline\cM_\cT (S)$ has
ergodic measures.

We put an accessory Riemannian structure on $S$. Then
theorem \ref{thm:transverse-riemannian} allows to identify
$\overline\cM_\cT (S)$ to $\cM_\cL(S)$. By proposition
\ref{prop:M_L-compact}, this is a compact convex set inside
the locally convex topological vector space of all (signed)
daval measures.
The statement now follows from the application of Krein-Milman
theorem. \hfill $\Box$

\medskip

\begin{definition}\label{def:volume}\textbf{\em (Volume of measured solenoids)}
For a measured Riemannian solenoid $S_\mu$ we define the {\rm volume
measure of $S$} as the unique probability measure (also denoted by
$\mu$) associated to the transversal measure $\mu=(\mu_T)$ by theorem
\ref{thm:transverse-riemannian}.
\end{definition}

For uniquely ergodic Riemannian solenoids $S$, this volume measure
is uniquely determined by the Riemannian structure (as for a compact
Riemannian manifold). We observe that, contrary to what happens with
compact manifolds, there is no canonical total mass normalization of
the volume of the solenoid depending only on the Riemannian metric.
This is the reason why we normalize $\mu$ to be a probability
measure.

\medskip

The following result generalizes the decomposition of any
measure on a smooth manifold into an absolutely continuous part
with respect to a Lebesgue measure and a singular part. Indeed, theorem \ref{thm:6.11--}
generalizes that decomposition to solenoids, since it
reduces to the classical result when the solenoid is a manifold.

We first define irregular measures. These are measures
which have no mass that disintegrates as volume along leaves.

\begin{definition}
Let $\mu$ be a measure supported on $S$. We say that $\mu$ is
{\rm irregular} if for any Borel set $A\subset S$ and for any non-zero
measure $\nu\in \cM_\cL (S)$ we do not have
 $$
 \nu_{|A} \leq \mu_{|A} \, .
 $$
\end{definition}

\begin{theorem} \label{thm:6.11--}
Let $\mu$ be any measure supported on $S$. There is a unique
canonical decomposition of $\mu$ into a regular part
$\mu_r\in\cM_\cL (S)$ and an irregular part $\mu_i$,
 $$
 \mu=\mu_r +\mu_i \, .
 $$
We can define the regular part by
 $$
 \mu_r (A)=\sup_\nu \nu (A) \leq \mu (A) \, ,
 $$
for any Borel set $A\subset S$, where the supremum runs over all
measures $\nu\in \cM_\cL (S)$, with $\nu_{|A} \leq \mu_{|A}$ (if no
such measure exists then $\mu_r (A)=0$).
\end{theorem}

\begin{proof}
Consider all measures $\nu\in\cM_\cL (S)$ such that $\nu\leq \mu$.
We define $\mu_r=\sup \nu$. Considering a countable basis
$(A_i)$ for the Borel $\sigma$-algebra and extracting a triangular
subsequence, we can find a sequence of such measures $(\nu_n)$ such
that $\nu_n (A_i)\to \mu_r(A_i)$, for all $i$, i.e.\ $\nu_n\to
\mu_r$. Since $\cM_\cL(S)$ is closed it follows that
$\mu_r\in\cM_\cL (S)$. By construction, $\mu-\mu_r$ is a positive
measure and irregular. Moreover the decomposition
 $$
 \mu=\mu_r+\mu_i
 $$
is unique, because for another decomposition
 $$
 \mu=\nu_r+\nu_i \, ,
 $$
we have by construction of $\mu_r$,
 $$
 \nu_r\leq \mu_r \, .
 $$
Therefore
 $$
 \nu_i =(\mu_r-\nu_r)+\mu_i \, ,
 $$
and $\mu_i$ being positive this implies that
 $$
 0\leq \mu_r- \nu_r \leq \nu_i \, .
 $$
By definition of irregularity of $\nu_i$, this is only possible if
$\mu_r=\nu_r$, then also $\mu_i=\nu_i$, and the decomposition is
unique.
\end{proof}

\section{Generalized Ruelle-Sullivan currents}\label{sec:Ruelle-Sullivan}

Our purpose in this section is to associate natural currents to
immersed solenoids. We fix in this section a $C^\infty$ manifold $M$
of dimension $n$.

\begin{definition}\label{def:solenoid-in-manifold}
\textbf{\em (Immersion and embedding of solenoids)} Let $S$ be a
$k$-solenoid of class $C^{r,s}$ with $r\geq 1$. A map $f:S \to M$ is
{\rm regular} if it has an extension $\hat{f}:W\to M$ of class
$C^{r,s}$, where $W$ is a foliated manifold which defines the
solenoid structure of $S$. An {\rm immersion}
 $$
 f:S\to M
 $$
is a regular map such that the differential restricted to the
tangent spaces of leaves has rank $k$ at every point of $S$. We say
that $f:S\to M$ is an {\rm immersed solenoid}.


Let $r,s\geq 1$. A {\rm transversally immersed solenoid} $f:S\to M$
is a regular map $f:S\to M$ such that
\begin{itemize}

\item It admits an extension $\hat{f}:W\to M$ which is an immersion of
a $(k+\ell)$-dimensional manifold into an $n$-dimensional one of
class $C^{r,s}$.

\item the images of the leaves intersect transversally in $M$.

\end{itemize}

An {\rm embedded solenoid} $f:S\to M$ is a transversally immersed
solenoid with injective $f$, that is, the leaves do not intersect or
self-intersect. Equivalently, $f:S\to M$ admits an extension
$\hat{f}:W\to M$ which is an embedding.
\end{definition}

Note that under a transversal immersion, resp.\ an embedding,
$f:S\to M$, the images of the leaves are immersed, resp.\
injectively immersed, submanifolds.

A foliation of $M$ can be considered as a solenoid, and the
identity map is an embedding.

We shall denote the space of compactly supported currents of
dimension $k$ by $\cC_k(M) $. These currents are functionals
$T:\Omega^k(M)\to \RR$. The space $\cC_k(M)$ is endowed with the
weak-* topology. A current $T\in \cC_k(M)$ is closed if
$T(d\alpha)=0$ for any $\alpha\in \Omega^{k-1}(M)$, i.e. if it
vanishes on $B^k(M)=\im (d:\Omega^{k-1} (M)\to \Omega^{k}(M))$.
Therefore, by restricting to $Z^k(M)=\ker (d:\Omega^k (M)\to
\Omega^{k+1}(M))$, a closed current $T$ defines a linear map
  $$
 [T]: H^k(M,\RR)=\frac{Z^k(M)}{B^k(M)} \longrightarrow \RR\, .
 $$
By duality, $T$ defines a
real homology class $[T]\in H_k(M,\RR)$.

We now define associated currents to immersions of solenoids. We use
in the definition a given measurable partition of unity and show
after the definition that the construction is independent of the
choice.

\begin{definition}\label{def:Ruelle-Sullivan}\textbf{\em (Generalized currents)}
Let $S$ be an oriented  $k$-solenoid of class $C^{r,s}$, $r\geq 1$,
endowed with a transversal measure $\mu=(\mu_T)$. An immersion
 $$
 f:S\to M
 $$
defines a current $(f,S_\mu)\in \cC_k(M)$, called {\rm generalized
Ruelle-Sullivan current}, or just {\rm generalized current}, as
follows.

Let $\omega$ be an $k$-differential form in $M$. The pull-back
$f^* \omega$ defines a $k$-differential form on the leaves of $S$.
Let $S=\bigcup_i S_i$ be a measurable partition such that each
$S_i$ is contained in a flow-box $U_i$.  We define
 $$
 \la (f,S_\mu),\omega \ra=\sum_i \int_{K(U_i)} \left ( \int_{L_y\cap S_i}
 f^* \omega \right ) \ d\mu_{K(U_i)} (y) \, ,
 $$
where $L_y$ denotes the horizontal disk of the flow-box.

The current $(f,S_\mu)$ is closed, hence it defines a
real homology class
 $$
 [f,S_\mu]\in H_k(M,\RR )\, ,
 $$
called Ruelle-Sullivan homology class.
\end{definition}

Note that this definition does not depend on the measurable
partition (given two partitions consider the common refinement).
If the support of $f^*\omega$ is contained in a flow-box $U$ then
 $$
 \la (f,S_\mu),\omega \ra =\int_{K(U)} \left ( \int_{L_y} f^* \omega \right )
 \ d\mu_{K(U)} (y) \, .
 $$
In general, take a partition of unity $\{\rho_i\}$ subordinated to
the covering $\{U_i\}$, then
  $$
   \la (f,S_\mu),\omega \ra = \sum_i
   \int_{K(U_i)} \left( \int_{L_y} \rho_i f^* \omega \right)
   d\mu_{K(U_i)} (y) \, .
  $$

Let us see that $(f,S_\mu)$ is closed.
For any exact differential $\omega=d\a$ we have
 $$
 \begin{aligned}
 \la (f,S_\mu),d\a\ra =\, &  \sum_i
 \int_{K(U_i)} \left ( \int_{L_y} \rho_i \, f^* d\a \right )
 \ d\mu_{K(U_i)}(y) \\  =\, &  \sum_i
 \int_{K(U_i)} \left ( \int_{L_y} d (\rho_i f^* \a) \right )
 \ d\mu_{K(U_i)}(y)   \\ & \qquad  -
  \sum_i \int_{K(U_i)} \left ( \int_{L_y} d \rho_i \wedge f^* \a \right )
 \ d\mu_{K(U_i)}(y)    = 0\, . \qquad
 \end{aligned}
 $$
The first term vanishes using Stokes in each leaf (the form $\rho_i f^* \a$
is compactly
supported on $U_i$), and the second term vanishes because $\sum_i d\rho_i\equiv 0$.
Therefore $[f, S_\mu]$ is a well defined homology class of degree $k$.

In their original article \cite{RS}, Ruelle and Sullivan defined
this notion for the restricted class of solenoids embedded in $M$.

\section{De Rham cohomology of solenoids}\label{sec:appendix2-generalities}

In this section we present the definition of the De Rham
cohomology groups for solenoids. The general theory for foliated
spaces from \cite[chapter 3]{MoSch} can be applied to our
solenoids. In \cite{MoSch}, the required regularity is of class $C^{\infty,0}$,
but it is easy to see that the arguments extend to the case of regularity
of class $C^{1,0}$.

Let $S$ be a $k$-solenoid of class $C^{r,s}$ with $r\geq 1$. The space of
$p$-forms $\Omega^p(S)$ consists of $p$-forms on leaves $\a$, such
that $\a$ and $d\a$ are of class $C^{1,0}$. 
Note that the exterior differential
 $$
 d:\Omega^p(X) \to \Omega^{p+1}(X)
 $$
is the differential in the leaf directions. We can define the De
Rham cohomology groups of $S$ as usual,
  $$
  H^p_{DR}(S):= \frac{\ker (d:\Omega^p(S)\to \Omega^{p+1}(S))}{\im
  (d: \Omega^{p-1}(S)\to \Omega^{p}(S))}\, .
  $$
The natural topology of the spaces $\Omega^p(X)$ gives a topology
on $H^p_{DR}(S)$, so this is a topological vector space, which is
in general non-Hausdorff. Quotienting by $\overline{\{0\}}$, the
closure of zero, we get a Hausdorff space
  $$
  \bar H^p_{DR}(S)= \frac{H^p_{DR}(S)}{\overline{\{0\}}} =
  \frac{\ker (d:\Omega^p(S)\to \Omega^{p+1}(S))}{\ \overline{\im
  (d: \Omega^{p-1}(S)\to \Omega^{p}(S))}\ }\, .
  $$

We define the solenoid homology as
 $$
 H_k(S):=\Hom_{cont}(H^k_{DR}(S),\RR)= \Hom_{cont}(\bar H^k_{DR}(S),\RR)\, ,
 $$
the continuous homomorphisms from the cohomology to $\RR$.

\begin{remark}
For a manifold $M$, $H^k_{DR}(M)$ and $H_k(M)$ are equal to the
usual cohomology and homology with real coefficients.
\end{remark}

\begin{definition}{\textbf{\em (Fundamental class)}}
\label{def:B.1} Let $S$ be an oriented $k$-solenoid with a
transversal measure $\mu=(\mu_T)$. Then there is a well-defined map
given by integration of $k$-forms
 $$
 \int_{S_\mu}  :\Omega^k(S) \to \RR\, ,
 $$
whcih assigns to any $\alpha\in \Omega^k(S)$ the number
 $$
  \int_{S_\mu}\alpha :=\sum_i \int_{K(U_i)} \left( \int_{L_y\cap S_i}
 \alpha(x,y) dx \right) d\mu_{K(U_i)} (y)\, ,
 $$
where $S_i$ is a finite measurable partition of $S$ subordinated
to a cover $\{U_i\}$ by flow-boxes. It is easy to see, as in section
\ref{sec:Ruelle-Sullivan}, that $\int_{S_\mu} d\beta
=0$ for any $\beta\in \Omega^{k-1}(S)$. Hence $\int_{S_\mu}$ gives
a well-defined map
 $$
 H^k_{DR}(S) \to \RR\, .
 $$
Moreover, this is a continuous linear map, hence it defines an
element
 $$
 [S_\mu]\in H_k(S)\, .
 $$
We shall call $[S_\mu]$ the {\rm fundamental class} of $S_\mu$.
\end{definition}

The following result is in \cite[theorem 4.27]{MoSch}. See also
\cite[theorem 2]{extra}.

\begin{theorem} \label{thm:Hk}
 Let $S$ be a compact, oriented $k$-solenoid. Then the map
  $$
  \cV_\cT(S) \to H_k(S)\, ,
  $$
 which sends $\mu$ to $[S_\mu]$, is an isomorphism from the
 space of all signed transversal measures to the $k$-homology of $S$.
\end{theorem}

The set of transversal measures $\cM_\cT(S)$ is a cone, which
generates $\cV_\cT(S)$. Its extremal points are the ergodic
transversal measures. These ergodic measures are linearly
independent. Therefore, the dimension of $H_k(S)$ coincides with
the number of ergodic transversal measures of $S$. Hence, if $S$
is uniquely ergodic, then $H_k(S)\cong \RR$, and $S$ has a unique
fundamental class (up to scalar factor). The uniquely ergodic
solenoids are the natural extension of compact manifolds without
boundary. For a compact, oriented, uniquely ergodic $k$-solenoid
$S$, there is a (Poincar\'{e} duality) coupling, 
  $$
  H^d_{DR}(S)\otimes H^{k-d}_{DR}(S) \to H^k_{DR}(S)
  \stackrel{\int_{S_\mu}}{\too} \RR\, ,
  $$
where $\mu$ is the transversal measure (unique up to scalar). See
\cite{extra} for the study of this.

The relationship of the fundamental class of a measured solenoid
$S_\mu$ with the Ruelle-Sullivan homology classes defined
by an immersion $f$ is given by the following result. 

\begin{proposition}\label{prop:RS-push-forward}
  Let $S_\mu$ be an oriented measured $k$-solenoid. If $f:S\to M$
  is an immersion, we have
    $$
    f_*[S_\mu]=[f,S_\mu] \in H_k(M,\RR) \, .
    $$
\end{proposition}

\begin{proof}
The equality $\la [f, S_\mu],\omega \ra= \la [S_\mu],
f^*\omega\ra$ is clear for any $\omega\in \Omega^k(M)$ (see the
construction of the fundamental class in definition
\ref{def:B.1}). The result follows.
\end{proof}

\medskip

We shall need some basic results about bundles over solenoids.
A real vector bundle of rank $m$ over a solenoid $S$ is defined as
follows. A rank $m$ vector bundle over a pre-solenoid $(S,W)$ (see
definition \ref{def:k-solenoid}) is a rank $m$ vector bundle
$\pi:E_W\to W$ whose transition functions
 $$
 g_{\a\b}: U_\a\cap U_\b \to \GL(m,\RR)
 $$
are of class $C^{r,s}$. We denote $E=\pi^{-1}(S)$, so that there
is a map $\pi:E\to S$. Let $(S,W_1)$ and $(S,W_2)$ be two
equivalent pre-solenoids, with $f:U_1\to U_2$ a diffeomorphism of
class $C^{r,s}$, $S\subset U_1\subset W_1$, $S\subset U_1\subset
W_1$ and $f_{|S}=\id$, then we say that $\pi_1:E_{W_1}\to W_1$ and
$\pi_2:E_{W_2}\to W_2$ are equivalent if
$\pi_1^{-1}(S)=\pi_2^{-1}(S)=E$ and there exists a vector bundle
isomorphism $\hat{f}:E_{W_1}\to E_{W_2}$ covering $f$ such that
$\hat{f}$ is the identity on $E$. Finally a vector bundle
$\pi:E\to S$ over $S$ is defined as an equivalence class of such
vector bundles $E_W\to W$ by the above equivalence relation.

Note that the total space $E$ of a rank $m$ vector bundle over a
$k$-solenoid $S$ inherits the structure of a $(k+m)$-solenoid
(although non-compact).

A vector bundle $E\to S$ is oriented if each fiber
$E_p=\pi^{-1}(p)$ has an orientation in a continuous manner. This
is equivalent to ask that there exist a representative $E_W\to W$
(where $W$ is a foliated manifold defining the solenoid structure
of $S$) which is an oriented vector bundle over the
$(k+\ell)$-dimensional manifold $W$.

\medskip

Let $S$ be a solenoid of class $C^{r,s}$ with $r\geq 1$,
and let $E\to S$ be a vector bundle. We may define forms on the total space $E$. A
form $\a\in \Omega^p(E)$ is of vertical compact support if the
restriction to each fiber is of compact support. The space of such
forms is denoted by $\Omega^p_{cv}(E)$. Note that this condition is
preserved under differentials, so it makes sense to talk about the
cohomology with vertical compact supports,
 $$
 H_{cv}^p(E) =\frac{\ker (d:\Omega^p_{cv}(E)\to \Omega^{p+1}_{cv}(E))}{\im
  (d: \Omega^{p-1}_{cv}(E)\to \Omega^{p}_{cv}(E))}\,  .
 $$

\begin{definition} \textbf{\em (Thom form)}
A {\rm Thom form} for an oriented vector bundle $E\to S$ of rank $m$
over a solenoid $S$ is an $m$-form
 $$
 \Phi\in \Omega_{cv}^m (E)\, ,
 $$
such that $d\Phi=0$ and $\Phi|_{E_p}$ has integral $1$ for each
$p\in S$ (the integral is well-defined since $E$ is
oriented).
\end{definition}

By the results of \cite{MoSch}, Thom forms exist. They represent a
unique class in $H_{cv}^m(E)$, i.e. if $\Phi_1$ and $\Phi_2$ are
two Thom forms, then there is a $\beta\in \Omega_{cv}^{m-1}(E)$
such that $\Phi_2-\Phi_1=d\beta$. Moreover, the map
  $$
  H^{k}(S) \to H^{m+k}_{cv}(E)
  $$
given by
  $$
  [\alpha] \mapsto [\Phi\wedge\pi^*\alpha]\, ,
  $$
is an isomorphism.

\section{Forms representing the Ruelle-Sullivan homology class}\label{sec:forms}

We make the simplifying assumption that the manifold
$M$ is compact and oriented of dimension $n$. We will make comments later
on the general case. Let $f:S_\mu\to M$ be an oriented measured
$k$-solenoid immersed in $M$. The Ruelle-Sullivan
homology class $[f,S_\mu]\in H_k(M,\RR)$ gives an element
 $$
 [f,S_\mu]^* \in H^{n-k}(M,\RR)\, ,
 $$
under the Poincar\'{e} duality isomorphism $H_k(M,\RR)\cong
H^{n-k}(M,\RR)$.
In this section, we shall construct a $(n-k)$-form representing
$[f,S_\mu]^*$.

\medskip

Fix an accessory Riemannian metric $g$ on $M$. This endows $S$
with a solenoid Riemannian metric $f^*g$. We can define the normal
bundle
 $$
 \pi:\nu_f \to S\, ,
 $$
which is an oriented bundle of rank $n-k$, since both $S$ and $M$
are oriented. The total space $\nu_f$ is a (non-compact)
$n$-solenoid whose leaves are the preimages by $\pi$ of the leaves
of $S$.

By section \ref{sec:appendix2-generalities}, there is a Thom form $\Phi\in
\Omega^{n-k}_{cv}(\nu_f)$ for the normal bundle. This is a closed
$(n-k)$-form on the total space of the bundle $\nu_f$, with
vertical compact support, and satisfying that
 $$
 \int_{\nu_{f,p}}\Phi=1\, ,
 $$
for all $p\in S$, where $\nu_{f,p}=\pi^{-1}(p)$. Denote by
$\nu_r\subset \nu$ the disc bundle formed by normal vectors of
norm at most $r$ at each point of $S$. By compactness of $S$,
there is an $r_0>0$ such that $\Phi$ has compact support on
$\nu_{r_0}$.

For any $\lambda>0$, let $T_\lambda:\nu_f\to \nu_f$ be the map
which is multiplication by $\lambda^{-1}$ in the fibers. Then set
  $$
  \Phi_r=T_{r/r_0}^*\Phi\, ,
  $$
for any $r>0$. So $\Phi_r$ is a closed $(n-k)$-form, supported in
$\nu_r$, and satisfying
 $$
 \int_{\nu_{f,p}}\Phi_r=1\, ,
 $$
for all $p\in S$. Hence it is a Thom form for the bundle $\nu_f$
as well. By section \ref{sec:appendix2-generalities}, $[\Phi_r]=[\Phi]$ in
$H^{n-k}_{cv}(\nu_f)$, i.e. $\Phi_r-\Phi=d\beta$, with $\beta\in
\Omega_{cv}^{n-k-1}(\nu_f)$.

\medskip

Using the exponential map and the immersion $f$, we have a map
 $$
 j:\nu_f \to M\, ,
 $$
given as $j(p,v)=\exp_{f(p)}(v)$, which is a regular map from the
$\nu_f$ (as an $n$-solenoid) to $M$. By compactness of $S$, there are
$r_1,r_2>0$ such that for any disc $D$ of radius $r_2$ contained
in a leaf of $S$, the map $j$ restricted to $\pi^{-1}(D)\cap
\nu_{r_1}$ is a diffeomorphism onto an open subset of $M$. 

\begin{proposition} \label{prop:push-forward}
There is a well defined push-forward linear map
  $$
  j_* : \Omega^{p}_{cv}(\nu_{r_1}) \to \Omega^p(M)\, ,
  $$
such that $dj_*\a=j_*d\a$, and $j_* (\a\wedge \b)= j_*\a \wedge
j_*\beta$, for $\a,\b \in \Omega^{p}_{cv}(\nu_{r_1})$.
\end{proposition}

\begin{proof}
Consider first a flow-box $U\cong D^k\x K(U)$ for $S$, where the
leaves of the flow-box are contained in discs of radius $r_2$.
Then
 $$
 \pi^{-1}(U) \cap \nu_{r_1} \cong D^{n-k}_{r_1} \x D^k \x K(U) \,
 ,
 $$
where $D^{n-k}_r$ denotes the disc of radius $r>0$ in $\RR^{n-k}$.
Let $\alpha\in \Omega^{p}_{cv}(\nu_{r_1})$ with support in
$\pi^{-1}(U) \cap \nu_{r_1}$. Then we define
 $$
 j_*\alpha := \int_{K(U)} \big( (j_y)_* (\a|_{D^{n-k}_{r_1}\x D^k \x
 \{y\}} ) \big) d\mu_{K(U)}(y)\, ,
 $$
where $j_y$ is the restriction of $j$ to ${D^{n-k}_{r_1}\x D^k \x
\{y\}} \subset \pi^{-1}(U)\cap \nu_{r_1}$, which is a
diffeomorphism onto its image in $M$. This is the average of the
push-forwards of $\alpha$ restricted to the leaves of $\nu_f$,
using the transversal measure.

Now in general, consider a covering $\{U_i\}$ of $S$ by flow-boxes
such that the leaves of the flow-boxes $U_i$ are contained in
discs of radius $r_2$. Then, for any form $\a\in
\Omega_{cv}^p(\nu_{r_1})$, we decompose $\a=\sum \a_i$ with $\a_i$
supported in $\pi^{-1}(U_i) \cap \nu_{r_1}$. Define
 $$
 j_*\a:=\sum j_*\a_i\, .
 $$
This does not depend on the chosen cover.

Finally, $j_*d\a=dj_*\a$ holds in flow-boxes, hence it holds globally.
The other assertion is analogous.
\end{proof}

Finally, we can construct a $(n-k)$-form representing $[f,S_\mu]^*$.

\begin{proposition} 
\label{prop:RS-form}
  Let $M$ be a compact oriented manifold.
  Let $f:S_\mu\to M$ be an oriented measured solenoid immersed in $M$.
  Let $\Phi_r$ be the Thom form of the normal bundle $\nu_f$
  supported on $\nu_r$, for $0<r<r_1$.
  Then  $j_*\Phi_r$ is a closed $(n-k)$-form representing the dual
  of the Ruelle-Sullivan homology class,
   $$
   [j_*\Phi_r]= [f,S_\mu]^*\, .
   $$
\end{proposition}

\begin{proof}
As $\Phi_r$ is a closed form, we have
 $$
 dj_*\Phi_r=j_*d\Phi_r=0\, ,
 $$
for $0<r\leq r_1$, so the class $[j_*\Phi_r]\in H^{n-k}(M,\RR)$ is
well-defined.

Now let $r,s$ such that $0<r\leq s<r_1$. Then $[\Phi_r]=[\Phi_s]$
in $H_{cv}^{n-k}(\nu_f)$, so there is a vertically compactly
supported $(n-k-1)$-form $\eta$ with
 \begin{equation}\label{eqn:BETA}
 \Phi_r-\Phi_{s}=d\eta\, .
 \end{equation}
Let $r_3>0$ be such that $\eta$ has support on $\nu_{r_3}$. We can
define a smooth map $F$ which is the identity on $\nu_s$, which
sends $\nu_{r_3}$ into $\nu_{r_1}$ and it is the identity on $\nu_f
- \nu_{2r_3}$. Pulling back (\ref{eqn:BETA}) with $F$, we get
 $$
 \Phi_r-\Phi_s=d (F^*\eta)\, ,
 $$
where $F^*\eta \in \Omega_{cv}^{n-k-1}(\nu_{r_1})$. We can apply
$j_*$ to this equality to get
 $$
 j_*\Phi_r-j_*\Phi_s=d j_*(F^*\eta)\, ,
 $$
and hence $[j_*\Phi_r]=[j_*\Phi_s]$ in $H^{n-k}(M,\RR)$.

\medskip

Now we want to prove that $[j_*\Phi_r]$ coincides with the dual of
the Ruelle-Sullivan homology class $[f,S_\mu]^*$. Let $\beta$ be any $k$-form
in $\Omega^k(S)$. Consider a cover $\{U_i\}$ of $S$ by flow-boxes
such that the leaves of each flow-box are contained in discs of
radius $r_2$, and let $\{\rho_i\}$ be a partition of unity
subordinated to this cover. Let $\Phi_i=\rho_i\Phi$, which is
supported on $\pi^{-1}(U_i)\cap \nu_f$, and
 $$
 \Phi_{r,i}=\rho_i\Phi_r= T^*_{r/r_0} \Phi_i
 $$
supported on $\pi^{-1}(U_i)\cap \nu_r$. For $0<\epsilon\leq r_1$,
we have
 $$
 \begin{aligned}
  \int_M j_*\Phi_{\epsilon,i} \wedge \beta &=
  \int_M  \left(  \int_{K(U_i)} (j_{i,y})_* (\Phi_{\epsilon,i|A_y^\epsilon})
  \, d\mu_{K(U_i)}(y)\right)  \wedge\beta \\
  &=   \int_{K(U_i)} \left( \int_M (j_{i,y})_* (\Phi_{\epsilon,i|A_y^\epsilon})
  \wedge\beta \right)\,
  d\mu_{K(U_i)}(y) \\
  &=  \int_{K(U_i)} \left( \int_{A_y^\epsilon} \Phi_{\epsilon,i}\wedge
  j_{i,y}^*\beta \right)\, d\mu_{K(U_i)}(y)\, ,
  \end{aligned}
 $$
where $A_y^\epsilon=D^{n-k}_\epsilon \x D^k\x \{y\}\subset
\pi^{-1}(U_i)$ for $y\in K(U_i)$, and $j_{i,y}=j_{|A_y^\epsilon}$.

In coordinates $(v_1,\ldots, v_{n-k},x_1,\ldots, x_k,y)$ for
$\pi^{-1}(U_i)\cong \RR^{n-k}\x D^k\x K(U_i)$, we can write
 $$
 \Phi=\Phi(v,x,y)= g_0 \,dv_{1}\wedge \cdots \wedge dv_{n-k} +
 \sum_{|I|>0} g_{IJ}\, dx_I\wedge dv_J\, ,
 $$
where $g_0$, $g_{IJ}$ are functions, and $I=\{i_1,\ldots,
i_a\}\subset \{1,\ldots, n-k\}$ and $J=\{j_1,\ldots, j_b\}\subset
\{1,\ldots, k\}$ multi-indices with $|I|=a$, $|J|=b$, $a+b=n-k$.
Pulling-back via $T=T_{\epsilon/r_0}$, we get
 \begin{equation}\label{eqn:Phi}
 \begin{aligned}
 \Phi_\epsilon &= \left(\frac{\epsilon}{r_0}\right)^{-(n-k)}
 \left( (g_0\circ T)
 \,dv_{1}\wedge \cdots \wedge dv_{n-k} + \sum_{|I|>0} \left(\frac{\epsilon}{r_0}\right)^{|I|} (g_{IJ}\circ T)
 \,  dx_I\wedge dv_J\right)  \\
 &=\left(\frac{\epsilon}{r_0}\right)^{-(n-k)}
 \left( (g_0\circ T) \,dv_{1}\wedge \cdots \wedge dv_{n-k} + O(\epsilon)\right) \,
 ,
 \end{aligned}
  \end{equation}
since $|g_{IJ}\circ T|$ are uniformly bounded. Note that the support
of $\Phi_{\epsilon|\RR^{n-k}\x D^k\x \{y\}}$ is inside
$D^{n-k}_\epsilon \x D^k\x \{y\}$.

Also write
 $$
 j_{i,y}^*\beta (v,x) =
 h_0(x,y)\, dx_1\wedge \ldots \wedge dx_k + \sum_{|J|>0}
 h_{IJ}(x,y)\, dx_I\wedge dv_J + O(|v|) \, ,
 $$
and note that $f^*\beta_{|D^k\x \{y\}} = h_0(x,y)\, dx_1\wedge
\ldots \wedge dx_k$.

So
 $$
  \begin{aligned}
  \int_{A_y^\epsilon} & \Phi_{\epsilon,i}\wedge j_{i,y}^*\beta =
  \int_{\RR^{n-k}\x D^k} \rho_i \Phi_{\epsilon}\wedge j_{i,y}^*\beta \\ &=
  \left(\frac{\epsilon}{r_0}\right)^{-(n-k)} \left( \int_{\RR^{n-k} \x D^{k}}
  \rho_i (g_0\circ T) \,dv_{1}\wedge \cdots \wedge dv_{n-k} \wedge
  j_{i,y}^*\beta + O(\epsilon^{n-k+1} ) \right) \\ &=
  \left(\frac{\epsilon}{r_0}\right)^{-(n-k)}  \bigg( \int_{\RR^{n-k}\x D^{k}}
  \rho_i (g_0\circ T) \,dv_{1} \wedge \cdots \wedge dv_{n-k} \wedge
  (h_0\, dx_1\wedge \ldots \wedge dx_k + O(|v|)) \bigg) +  O(\epsilon)
  \\ &=
  \left(\frac{\epsilon}{r_0}\right)^{-(n-k)}  \left( \int_{\RR^{n-k}\x D^{k}}
  \rho_i h_0 (g_0\circ T) \,dv_{1} \wedge \cdots \wedge dv_{n-k} \wedge
  dx_1\wedge \ldots \wedge dx_k  \right)+  O(\epsilon) \\ &=
  \int_{D^{k}}
  \rho_i h_0 \,dx_1\wedge \ldots \wedge dx_k   +  O(\epsilon) \\ &=
  \int_{D^k\x \{y\}} \rho_i  \,
  f^*\beta_{|D^k\x \{y\}}  + O(\epsilon)\, .
  \end{aligned}
 $$
The second equality holds since $|\rho_i|\leq 1$, $|j_{i,y}^*\beta|$
is uniformly bounded, and the support of $\rho_i (g_{IJ}\circ T) \,
dx_I\wedge dv_J \wedge j_{i,y}^*\beta$ is contained inside
$D^{n-k}_\epsilon \x D^k$, which has volume $O(\epsilon^{n-k})$. In
the fourth line we use that $|v|\leq \epsilon$ and
 $$
 \left(\frac{\epsilon}{r_0}\right)^{-(n-k)}  \int_{\RR^{n-k}}
 (g_0\circ T) \,dv_{1} \wedge \cdots \wedge dv_{n-k} =
 \int_{\nu_{f,p}} T^* \Phi = 1.
 $$
The same equality is used in the fifth line.

Adding over all $i$, we get
  $$
  \begin{aligned}
  \la [j_* \Phi_\epsilon], [\b]\ra =
  \int_M j_*\Phi_\epsilon\wedge \b 
  &= \sum_i
  \int_M j_*\Phi_{\epsilon,i} \wedge \beta \\ &=
  \sum_i  \int_{K(U_i)} \left( \int_{A_y^\epsilon} \Phi_{\epsilon,i}\wedge
  j_{i,y}^*\beta \right)\, d\mu_{K(U_i)}(y)
  \\ &= \sum_i \int_{K(U_i)} \left(  \int_{D^k\x \{y\}} \rho_i  \,
  f^*\beta|_{D^k\x \{y\}}  + O(\epsilon) \right)
  d\mu_{K(U_i)}(y) \\ &= \la [f, S_\mu], \b \ra + O(\epsilon)\, .
  \end{aligned}
  $$
Taking $\epsilon\to 0$, we get that
 $$
 [j_*\Phi_r] = \lim_{\epsilon\to 0}\  [j_* \Phi_\epsilon]=[f,S_\mu]^*\,
 ,
 $$
for all $0<r<r_1$.
\end{proof}

\medskip

\noindent {\bf Case $M$ non-compact.}


For $M$ non-compact, we have the isomorphism $H_k(M,\RR)\cong
H^{n-k}_c(M,\RR)$, where $H_c^*(M,\RR)$ denotes compactly
supported cohomology of $M$. Then the Ruelle-Sullivan homology class
$[f,S_\mu]$ of an immersed oriented measured
solenoid $f:S_\mu \to M$ gives an element
  $$
  [f,S_\mu]^* \in H^{n-k}_c(M,\RR)\, .
  $$
The construction of the proof of proposition \ref{prop:RS-form}
gives a smooth compactly supported form $j_*\Phi_r$ on $M$, for
$r$ small enough, with
  $$
  [j_*\Phi_r]=[f,S_\mu]^* \in H^{n-k}_c(M,\RR)\, .
  $$

\medskip

\noindent {\bf Case $M$ non-oriented.}


For $M$ non-oriented, let $\fro$ be the local system defining the
orientation of $M$. Let $f:S_\mu \to M$ be an immersed oriented
measured solenoid. Then both $[f,S_\mu]$ and $[j_*\Phi_r]$ are classes
which correspond under the isomorphism
 $$
 H_k(M,\RR) \cong H^{n-k}_c(M,\fro)\, .
 $$
The same proof as above shows that they are equal.

\bigskip

We end up this section with an easy remark on the case of a complex solenoid
immersed in a complex manifold.
If $M$ is a complex manifold, an immersed solenoid
$f:S\to M$ is complex if the leaves of $S$ have a (transversally
continuous) complex structure, and $f$ is a holomorphic immersion
on every leaf $l\subset S$. Note that $S$ is automatically
oriented.

\begin{proposition}\label{prop:complex}
Let $M$ be a compact complex K\"ahler manifold. Let $f:S_\mu\to M$ be
an embedded complex $k$-solenoid endowed with a transversal
measure. Then
 $$
 [f,S_\mu]^* \in H^{p,p}(M) \cap H^{2n-k}(M,\RR) \subset H^{2n-k}(M,\CC)\, ,
 $$
where $k=2(n-p)$.
\end{proposition}

\begin{proof}
 For a compact K\"ahler manifold, we have the Hodge decomposition
 of the cohomology into components of $(p,q)$-types,
  $$
  H^{2n-k}(M,\CC)= \bigoplus_{p+q=2n-k} H^{p,q}(M) \, .
  $$
 The statement of the
 proposition is equivalent to the vanishing of
  $$
  \la [f,S_\mu], [\alpha] \ra\, ,
  $$
 for any $[\alpha]\in H^{p,q}(M)$, $p+q=k$, with $p\neq q$. But
 for any
 $\alpha\in \Omega^{p,q}(M)$, we have that $f^*\alpha \equiv 0$. This is easy to
 see as follows: pick local coordinates $(w_1,\ldots, w_n)$ for $M$ and
 consider $\alpha=dw_{i_1}\wedge \ldots dw_{i_p}\wedge
 d\bar{w}_{j_1}\wedge \ldots d\bar w_{j_q}$. Suppose that $p>q$,
 so that $p>k'$, $k=2k'$. Locally
 $f$ is written as $f: U=D^{2k'} \x K(U)\to M$,
 $(w_1,\ldots, w_n)= f(z_1,\ldots,z_{k'}, y)$, with $f$ holomorphic
 with respect to $(z_1,\ldots, z_{k'})\in \CC^{k'}$. 
 Clearly $f^*\alpha$ contains $p$ differentials $dz_i$'s and $q$
differentials $d\bar z_j$'s.
 As $p>k'$, we have that $f^*\alpha=0$. The case $p<q$ is similar.
\end{proof}

\section{Self-intersection of embedded solenoids} \label{sec:embedded}

Let $M$ be a compact oriented manifold, and consider an \emph{embedded} oriented
measured solenoid $f:S_\mu\to M$. In this section, we want to prove that
the self-intersection of the generalized current is zero.

\begin{theorem}\label{thm:self-intersection} \textbf{ \em
(Self-intersection of embedded solenoids)} Let $M$ be a compact,
oriented, smooth manifold. Let $f:S_\mu\to M$ be an embedded oriented measured
solenoid, such that the transversal measures
$(\mu_T)$ have no atoms. Then we have
 $$
 [f,S_\mu]^* \cup [f,S_\mu]^*=0
 $$
in $H^{2(n-k)}(M,\RR)$.
\end{theorem}

\begin{proof}
If $n-k>k$ then $2(n-k)>n$, therefore the self-intersection is $0$ by degree reasons.
So we may assume $n-k\leq k$.

Let $\b$ be any closed $(n-2(n-k))$-form on $M$. We must
prove that
 $$
 \la [f,S_\mu]^*\cup [f,S_\mu]^*\cup [\b],[M]\ra =0\,,
 $$
where $[M]$ is the fundamental class of $M$.
By proposition \ref{prop:RS-form},
 $$
 \la [f,S_\mu]^*\cup [f,S_\mu]^*\cup [\b],[M]\ra =
 \la [f,S_\mu], j_*\Phi_\epsilon\wedge \b\ra\, ,
 $$
for $\epsilon>0$ small enough.

Consider a covering of $f(S)\subset M$ by open sets $\hat U_i\subset
M$ and another covering of $f(S)$ by open sets $\hat V_i\subset M$
such that the closure of $\hat{V}_i$ is contained in $\hat U_i$. We
may assume that the covering is chosen so that $\{V_i=f^{-1}(\hat
V_i) \}$ satisfies the properties needed for computing
$j_*\Phi_\epsilon$ locally (the auxiliary Riemannian structure is
used). Let $\{\rho_i\}$ be a partition of unity of $S$ subordinated to
$\{V_i \}$ and decompose $\Phi_\epsilon=\sum \Phi_{\epsilon,i}$ with
$\Phi_{\epsilon,i}=\rho_i \, \Phi_\epsilon$. We take $\epsilon>0$
small enough so that $j(\supp \Phi_{\epsilon,i})\subset \hat{U}_i$.
Then
 $$
  \la [f,S_\mu]^*\cup [f,S_\mu]^*\cup [\b],[M]\ra=
 \la [f,S_\mu], j_*\Phi_\epsilon\wedge \b\ra\
  =  \sum_i\la [f,S_\mu], j_*\Phi_{\epsilon,i} \wedge \b\ra\,
 .
 $$

Since $f$ is an embedding, we may suppose the open sets
$U_i=f^{-1}(\hat U_i)$ are flow-boxes of $S$. Therefore
 $$
 \la [f,S_\mu], j_*\Phi_{\epsilon,i} \wedge \b \ra =\int_{K(U_i)} \left(\int_{L_y}
 f^*( j_* \Phi_{\epsilon,i}\wedge \beta) \right) d\mu_{K(U_i)}(y)\,
 .
 $$
We may compute
 $$
 \begin{aligned}
 \int_{L_y} f^*(j_* \Phi_{\epsilon,i}\wedge \beta )&= \int_{f(L_y)}
 \left(\int_{K(V_i)}(j_{i,z})_*\Phi_{\epsilon,i}
 \wedge \beta \, d\mu_{K(V_i)}(z) \right) \\ &= \int_{K(V_i)} \left( \int_{f(L_y)}
 (j_{i,z})_*\Phi_{\epsilon,i} \wedge \beta \right) d\mu_{K(V_i)}(z)\, .
 \end{aligned}
 $$
Note that $(j_{i,z})_* \Phi_{i,\epsilon} |_{f(L_y)}$ consists of
restricting the form $\Phi_{i,\epsilon}$ to $\pi^{-1}(L_z)$, the
normal bundle over the leaf $L_z$, then sending it to $M$ via $j$,
and finally restricting to the leaf $f(L_y)$.

Since $f$ is an embedding, we may suppose that in a local chart
$f:U_i=D^k\x K(U_i) \to \hat U_i\subset M$ is the restriction of a
map (that we denote with the same letter) $f: D^k\x B \to \hat U_i$,
where $B\subset \RR^l$ is open and $K(U_i)\subset B$, which in suitable
coordinates for $M$ is written as $f(x,y)=(x,y,0)$. The map $j$
extends to  a map from the normal bundle to the horizontal foliation
of $D^k\x B$, as $j: D^{n-k}_\epsilon \x D^k \x B \to M$,
 $$
 j(v,x,z)= (x_1,\ldots, x_k, z_1+v_1,\ldots, z_l+v_l, v_{l+1},
 \ldots, v_{n-k}) +O(|v|^2)\, .
 $$

Using the formula for $\Phi_{\epsilon}$ given in \eqref{eqn:Phi}, we
have
 $$
  (j_{i,z})_* \Phi_\epsilon (x,y) = \sum_{|I|+|J|=n-k}
  \left(\frac{\epsilon}{r_0}\right)^{|I|-(n-k)}
  (g_{IJ}\circ T)(x, y-z) \,  dx_I\wedge dy_J + O(|y-z|)\, .
 $$
We restrict to $L_y$, and multiply by $\beta$, to get
 $$
 ((j_{i,z})_* \Phi_{\epsilon,i}\wedge \beta)|_{L_y} =
 \sum_{|I|=n-k} (\rho_i \cdot (g_{I0}\circ T))(x,y-z)
 \,  dx_I \wedge \beta
 + O(|y-z|) \, ,
 $$
which is bounded by a universal constant.

Hence
 $$
 |\la [f, S_\mu], j_*\Phi_{\epsilon,i}\wedge\b\ra| \leq
 C_0\ \mu_{K(U_i)}(K(U_i))\ \mu_{K(V_i)}(K(V_i))\leq C_0 \
 \mu_{K(U_i)}(K(U_i))^2 \,,
 $$
where $C_0$ is a constant that is valid for any refinement of the
covering $\{U_i\}$. So
 $$
 |\la [f, S_\mu], j_*\Phi_\epsilon \wedge\b \ra| \leq
 C_0 \sum_i \ \mu_{K(U_i)}(K(U_i))^2\,.
 $$

Observe that $\mu_{K(U_i)}(K(U_i))\leq C_1 \, \mu (U_i)$ and that
$\sum_i \mu (U_i)\leq C_2$ for some positive constants $C_1$ and
$C_2$ independent of the refinements of the covering. Therefore,
 $$
 \begin{aligned}
 |\la [f, S_\mu], j_*\Phi_\epsilon \wedge\b \ra| &\,\leq
 C_0  (\max_i \mu_{K(U_i)} (K(U_i)) )\ C_1 \sum_i \ \mu (U_i) \\
 &\,\leq C_0 C_1 C_2 \max_i \mu_{K(U_i)} (K(U_i))\,.
  \end{aligned}
 $$
When we refine the covering, if the transversal measures have no
atoms, we clearly have that $\max_i \mu_{K(U_i)} (K(U_i)) \to 0$ and then
$$
 \la [f, S_\mu], j_*\Phi_\epsilon \wedge\b \ra =0\,,
 $$
as required.
\end{proof}

Note that for a compact solenoid, atoms of transversal measures must
give compact leaves (contained in the support of the atomic part),
since otherwise at the accumulation set of the leaf we would have a
transversal $T$ with $\mu_T$ not locally finite. In particular if
$S$ is a minimal solenoid which is not a $k$-manifold, then all
transversal measures have no atoms. Therefore, the existence
of transversal measures with atomic part is equivalent to the
existence of compact leaves. This observation gives the following
corollary.

\begin{corollary}\label{cor:self-intersection-no-compact-leaves}
Let $M$ be a compact, oriented, smooth manifold. Let $f:S\to M$ be an
embedded oriented solenoid, such that $S$
has no compact leaves. Then for any tranversal measure $\mu$, we
have
 $$
 [f,S_\mu]^* \cup [f,S_\mu]^*=0
 $$
in $H^{2(n-k)}(M,\RR)$.
\end{corollary}

\medskip

\begin{remark}
We observe that if we want to represent a homology class $a\in
H_k(M,\RR)$ by an immersed solenoid in an $n$-dimensional manifold
$M$ and $a\cup a\neq 0$, then the solenoid cannot be embedded.
Note that when $n-k$ is odd, there is no
obstruction. We shall see in \cite{MPM4} that 
we can always obtain a transversally immersed solenoid representing $a$,
for any homology class $a \in H_k(M,\RR)$.
\end{remark}

\end{document}